\documentclass{article}
\usepackage{amssymb,amsmath}
\usepackage{graphicx,enumerate}
\usepackage[utf8]{inputenc}
\title{New Results on Congruence Boolean Lifting Property}

\author{George Georgescu \\ \footnotesize University of Bucharest\\ \footnotesize Faculty of Mathematics and Computer Science\\ \footnotesize Bucharest, Romania\\ \footnotesize Email: georgescu.capreni@yahoo.com}

\date{}

\begin{document}
\maketitle

\begin{abstract}
The Lifting Idempotent Property ($LIP$) of ideals in commutative rings inspired the study of Boolean lifting properties in the context of other concrete algebraic structures ($MV$-algebras, commutative l-groups, $BL$-algebras, bounded distributive lattices, residuated lattices,etc.), as well as for some types of universal algebras (C. Muresan and the author defined and studied  the Congruence Boolean Lifting Property ($CBLP$) for congruence modular algebras).

 A lifting ideal of a ring $R$  is an ideal of $R$ fulfilling $LIP$. In a recent paper, Tarizadeh and Sharma obtained new results on lifting ideals in commutative rings. The aim of this paper is to extend an important part of their results to congruences with $CBLP$ in semidegenerate congruence modular algebras. The reticulation of such algebra will play an important role in our investigations (recall that the reticulation of a congruence modular algebra $A$ is a bounded distributive lattice $L(A)$ whose prime spectrum is homeomorphic with Agliano's prime spectrum of $A$).

 Almost all results regarding $CBLP$ are obtained in the setting of semidegenerate congruence modular algebras having the property that the reticulations preserve the Boolean center.

 The paper contains several properties of congruences with $CBLP$. Among the results we mention a characterization theorem of congruences with $CBLP$. We achieve various conditions that ensure $CBLP$. Our results can be applied to a lot of types of concrete structures: commutative rings, $l$-groups, distributive lattices, $MV$-algebras, $BL$-algebras, residuated lattices, etc.

\end{abstract}

\textbf{Keywords}: commutator operation, semidegenerate congruence  modular algebras, prime congruences, reticulation, Boolean center, Congruence Boolean Lifting Property.
\newtheorem{definitie}{Definition}[section]
\newtheorem{propozitie}[definitie]{Proposition}
\newtheorem{remarca}[definitie]{Remark}
\newtheorem{exemplu}[definitie]{Example}
\newtheorem{intrebare}[definitie]{Open question}
\newtheorem{lema}[definitie]{Lemma}
\newtheorem{teorema}[definitie]{Theorem}
\newtheorem{corolar}[definitie]{Corollary}

\newenvironment{proof}{\noindent\textbf{Proof.}}{\hfill\rule{2mm}{2mm}\vspace*{5mm}}

\section{Introduction}

 \hspace{0.5cm} An ideal $I$ of a unital ring $R$ is said to be a lifting ideal if it fulfills the Lifting Idempotent Property ($LIP$): any idempotent of the quotient ring $R/I$ can be lifted to an idempotent of $R$ \cite{a}, \cite{Tar4}. If any ideal of $R$ is a lifting ideal then we say that $R$ has $LIP$. A vast literature was dedicated to lifting ideals and to rings that fulfill $LIP$ (see \cite{a}, \cite{Aghajani}, \cite{Al-Ezeh2}, \cite{b}, \cite{c}, \cite{Tar4}) and many important results for rings with $LIP$ were obtained. For example, in the case of commutative rings it was proved a notable theorem that asserts that the rings with $LIP$, the clean rings and the exchange rings coincide \cite{a}.

 Inspired by the theory of rings with $LIP$, various lifting properties were defined for other algebraic structures ($MV$-algebras \cite{Filipoiu}, commutative $l$-groups \cite{f}, $BL$-algebras \cite{DiNola}, pseudo $BL$-algebras \cite{B1} , \cite{B2}, bounded distributive lattices \cite{Cheptea}, residuated lattices \cite{GCM}, orthomodular lattices \cite{Kermani}, etc.). A lifting property, named Congruence Boolean Lifting Property ($CBLP$), was studied in a universal algebra framework: for congruence distributive algebras \cite{GM3} and for semidegenerate congruence modular algebras \cite{GKM}.

 In the recent paper \cite{Tar4}, Tarizadeh and Sharma obtained new algebraic and topological results on lifting ideals in a commutative ring (among them we mention an important characterization theorem for lifting ideals).

The aim of this paper is to extend some results of \cite{Tar4} to congruences with $CBLP$ in some types of universal algebras, so we continue the investigations on $CBLP$ started in \cite{GKM}.

The wonderful commutator theory, developed by Fresee and McKenzie in \cite{Fresee}, allows us to define a "prime spectrum" $Spec(A)$ of any congruence modular algebra $A$. According to Agliano's paper \cite{Agliano}, $Spec(A)$ fulfills remarkable topological properties. On the other hand, the hypothesis that the congruence modular algebra $A$ is semidegenerate causes the set $Con(A)$ of congruences of $A$ to be endowed with a structure of commutative and integral complete $l$-groupoid \cite{Birkhoff} (in fact, $Con(A)$ becomes an $mi$-structure in the sense of \cite{GeorgescuVoiculescu2}).

 Almost all results regarding $CBLP$ are obtained in the setting of semidegenerate congruence modular algebras having the property that the reticulation
  preserves the Boolean center.

 The paper contains several properties of congruences with $CBLP$. Among the results we mention a characterization theorem of congruences with $CBLP$. We achieve various conditions that assure $CBLP$. Our results can be applied to a lot of types of concrete structures: commutative rings, $l$-groups, distributive lattices, $MV$-algebras, $BL$-algebras, residuated lattices, etc.

 Now we shall describe the content of the paper. Section $2$ contains some preliminaries: the definition and elementary properties of commutator operation \cite{Fresee}, some basic facts on semidegenerate congruence modular algebras \cite{Agliano}, \cite{Kollar}, prime congruences, the radical of a congruence, Agliano's spectrum \cite{Agliano}.

 In Section $3$ we remember from \cite{GM2}, \cite{GG} the construction and the main properties of reticulation $L(A)$ associated with any semidegenerate congruence modular algebra $A$ that fulfills the following property: the set K(A) of compact congruences of $A$ is closed under the commutator. The reticulation $L(A)$ is a bounded distributive lattice whose prime spectrum $Spec(L(A))$ is homeomorphic with $Spec(A)$. This functorial construction was firstly done for commutative rings \cite{Simmons},\cite{Johnstone}, \cite{Al-Ezeh2}, then for many other algebraic structures. It is a good vehicle for exporting some properties from bounded distributive lattices to algebras and vice-versa (see  \cite{Al-Ezeh2},\cite{Dickmann}, \cite{DiNola}, \cite{Filipoiu}, \cite{GG}, \cite{Lenzi}, \cite{Muresan}, \cite{Simmons}). In this paper we shall use the reticulation for transporting some properties of lifting properties from rings to algebras.

 In Section $4$ we characterize the algebras whose reticulation preserves the Boolean center. We prove that the Boolean center $B(Con(A))$ of an algebra $A$ is isomorphic to the Boolean algebra $Clop(Spec(A))$ of clopen subsets of $Spec(A)$ (this results generalizes the Grothendieck correspondence between the idempotents of a ring and the clopens of its prime spectrum). The mentioned result (= Theorem 4.9) is used in Section $5$ to prove some properties of congruences with $CBLP$. We find the form of the clopen subsets of the maximal spectrum $Max(A)$ of the algebra $A$ and we characterize the situation whenever $Rad(A)$ has $CBLP$ ($Rad(A)= \bigcap Max(A)$ is a congruence that generalizes the Jacobson radical of a ring).

A characterization theorem of congruences with $CBLP$ is obtained in Section $6$ (see Theorem 6.3). The proof of this result is based on Hochster's theorem \cite{Hochster}, \cite{Dickmann}: for any bounded distributive lattice $L$ there exists a commutative ring $R$ such that the reticulation $L(R)$ of $R$ \cite{Simmons} is isomorphic with $L$.

Section $7$ concerns the relationship between lifting properties and orthogonal sets of complemented congruences. If a morphism $u$ of algebras lifts the complemented congruences then it is proven that $u$ lifts the countable orthogonal sets of complemented congruences. We investigate how some particular orthogonal sets of complemented congruences can be lifted by the morphisms of the form $A\rightarrow A/{\theta}$, where $\theta$ is an arbitrary congruence of the algebra $A$.

\section{Preliminaries}

 \hspace{0.5cm} Let $\tau$ be a finite signature of universal algebras. Throughout this paper we shall assume that the algebras have the signature $\tau$. Let $A$ be an algebra and $Con(A)$ the complete lattice of its congruences; $\Delta_A$ and $\nabla_A$ shall be the first and the last elements of $Con(A)$. If $X\subseteq A^2$ then $Cg_A(X)$ will be the congruence of $A$ generated by $X$; if $X = \{(a,b)\}$ with $a,b\in A$ then $Cg_A(a,b)$ will denote the (principal) congruence generated by $\{(a,b)\}$. $Con(A)$ is an algebraic lattice: the finitely generated congruences of $A$ are its compact elements. $K(A)$ will denote the set of compact congruences of $A$. We observe that $K(A)$ is closed under finite joins of $Con(A)$ and $\Delta_A\in K(A)$.

 For any $\theta\in Con(A)$, $A/\theta$ is the quotient algebra of $A$ w.r.t. $\theta$; if $a\in A$ then $a/\theta$ is the congruence class of $a$ modulo $\theta$. We shall denote by $p_\theta:A\rightarrow A/\theta$ the canonical surjective $\tau$ - morphism defined by $p_\theta(a) = a/\theta$, for all $a\in A$.

 Let  $\mathcal{V}$  be a congruence  modular variety of $\tau$ - algebras. Following \cite{Fresee}, p.31, the commutator is the greatest operation $[\cdot,\cdot]_A$ on the congruence lattices $Con(A)$ of members $A$ of $\mathcal{V}$ such that for any surjective morphism $f:A\rightarrow B$ of $\mathcal{V}$ and for any $\alpha,\beta\in Con(A)$, the following conditions hold

 (2.1) $[\alpha,\beta]_A\subseteq \alpha\cap \beta$;

 (2.2) $[\alpha,\beta]_A\lor Ker(f)$ = $f^{-1}([f(\alpha\lor Ker(f)),f(\beta\lor Ker(f))]_B)$.

 If $\alpha, \beta, \theta\in Con(A)$ then, by (2.2) we get

 (2.3) $([\alpha,\beta]_A\lor \theta)/\theta$ = $[(\alpha\lor \theta)/\theta,(\beta\lor \theta)/\theta]_{A/\theta}$.

 The commutator operation is commutative, increasing in each argument and distributive with respect to arbitrary joins. If there is no danger of confusion then we write $[\alpha,\beta]$ instead of $[\alpha,\beta]_A$.

 \begin{propozitie}\cite{Fresee}
 For any congruence  modular variety $\mathcal{V}$ the following are equivalent:
\newcounter{nr}
\begin{list}{(\arabic{nr})}{\usecounter{nr}}
\item  $\mathcal{V}$ has Horn - Fraser property: if $A,B$ are members of $\mathcal{V}$ then the lattices $Con(A\times B)$ and $Con(A)\times Con(B)$ are isomorphic;
\item $[\nabla_A,\nabla_A] = \nabla_A$, for all $A\in \mathcal{V}$;
\item $[\theta,\nabla_A] = \theta$, for all $A\in \mathcal{V}$ and $\theta\in Con(A)$.
\end{list}
\end{propozitie}

Following \cite{Kollar}, a variety $\mathcal{V}$ is semidegenerate if no nontrivial algebra in $\mathcal{V}$ has one - element subalgebras. By \cite{Kollar}, a variety $\mathcal{V}$ is semidegenerate if and only if for any algebra $A$ in $\mathcal{V}$, the congruence $\nabla_A$ is compact.

\begin{propozitie}\cite{Agliano}
If $\mathcal{V}$ is a semidegenerate congruence  modular variety then for each algebra $A$ in $\mathcal{V}$ we have $[\nabla_A,\nabla_A] = \nabla_A$.
\end{propozitie}

Let $A$ be a semidegenerate congruence  modular algebra. Therefore one can define on the complete lattice $Con(A)$ a residuation operation ( = implication) $\alpha \rightarrow \beta = \bigvee \{ \gamma|[\alpha,\gamma] \subseteq \beta\}$ and an annihilator operation ( = negation) $\alpha^{\bot } = \alpha^{\bot_A } =  \alpha \rightarrow \Delta_A =\bigvee \{\gamma |[\alpha,\gamma] = \Delta_A\}$. The implication $\rightarrow$ fulfills the usual residuation property: for all $\alpha,\beta,\gamma\in Con(A)$, $\alpha\subseteq \beta\rightarrow\gamma$ if and only if $[\alpha,\beta]\subseteq\gamma$. By using Propositions 2.1 and 2.2 we remark that $(Con(A), \lor, \land, [\cdot,\cdot], \rightarrow, \Delta_A, \nabla_A)$ is commutative and integral complete $l$ - groupoid (see \cite{Birkhoff}). In fact, $(Con(A),\lor,\cap,[\cdot,\cdot]_A,\Delta_A,\nabla_A)$ is a multiplicative - ideal structure (= mi - structure) in the sense of \cite{GeorgescuVoiculescu2}. Thus all the results contained in \cite{GeorgescuVoiculescu2} hold for the particular $mi$ - structure $Con(A)$.

For the rest of the section we fix an algebra $A$ in a semidegenerate congruence  modular variety $\mathcal{V}$.

\begin{lema}\cite{GM2}
 For all congruences $\alpha, \beta, \gamma $ the following hold:
\usecounter{nr}
\begin{list}{(\arabic{nr})}{\usecounter{nr}}
\item $\alpha\lor \beta = \nabla_A$ implies $[\alpha,\beta] = \alpha\cap \beta$;
\item $\alpha\lor \beta = \alpha\lor \gamma = \nabla_A$ implies $\alpha\lor [\beta,\gamma] = \alpha\lor (\beta\cap \gamma) = \nabla_A$;
\item $\alpha\lor \beta = \nabla_A$ implies $[\alpha,\alpha]^n\lor [\beta,\beta]^n = \nabla_A$, for all integers $n > 0$;
.
\end{list}
\end{lema}

For all congruences $\alpha, \beta \in Con(A)$ and for any integer $n\geq 1$ we define by induction the congruence $[\alpha,\beta]^n$: $[\alpha,\beta]^1$ = $[\alpha,\beta]$ and $[\alpha,\beta]^{n+1} =[[\alpha,\beta]^n,[\alpha,\beta]^n]$. By convention, we set $[\alpha,\alpha]^0 = \alpha$.

\begin{lema}
If $\alpha, \beta, \theta\in Con(A)$ and $\theta\subseteq \alpha\cap \beta$ then we have $[\alpha/\theta,\beta/\theta]^n$ = $([\alpha,\beta]^n\lor \theta)/\theta$.
\end{lema}

Following \cite{Fresee}, p.82 or \cite{Agliano}, p.582, a congruence $\phi\in Con(A)- \{\nabla_A \}$  is ${\emph{prime}}$ if for all $\alpha, \beta \in Con(A)$, $[\alpha,\beta] \subseteq \phi$ implies $\alpha \subseteq \phi$ or $\beta \subseteq \phi$. Let us introduce the following notations: $Spec(A)$ is the set of prime congruences and $Max(A)$ is the set of maximal elements of $Con(A)$. If $\theta \in Con(A)- \{\nabla_A \}$  then there exists $\phi \in Max(A)$ such that $\theta \subseteq \phi$ (because $\nabla_A$ is a compact congruence). By \cite{Agliano}, the following inclusion $Max(A) \subseteq Spec(A)$ holds. We shall denote by $Rad(A$) the intersection of all maximal congruences of $A$; $Rad(A)$ generalizes the notion of Jacobson radical of a ring.

According to \cite{Agliano}, p.582, the {\emph{radical}} $\rho(\theta)=\rho_A(\theta)$ of a congruence $\theta \in A$ is defined by $\rho_A(\theta)=\bigwedge \{\phi\in Spec(A)|\theta \subseteq \phi\}$; if $\theta=\rho(\theta)$ then $\theta$ is a radical congruence. We shall denote by $RCon(A)$ the set of radical congruences of $A$. The algebra $A$ is {\emph{semiprime}} if $\rho(\Delta_A)=\Delta_A$.

\begin{lema}
\cite{Agliano},\cite{GM2}
For all congruences $\alpha, \beta \in Con(A)$ the following hold:
\usecounter{nr}
\begin{list}{(\arabic{nr})}{\usecounter{nr}}
\item $\alpha \subseteq \rho(\alpha)$;
\item $\rho(\alpha \cap \beta)=\rho ([\alpha,\beta])=\rho(\alpha) \cap \rho(\beta)$;
\item $\rho(\alpha)= \nabla_A$ iff $\alpha = \nabla_A$;
\item $\rho(\alpha \lor \beta)=\rho(\rho(\alpha) \lor \rho(\beta))$;
\item $\rho(\rho(\alpha))=\rho(\alpha)$;
\item $\rho(\alpha) \lor \rho(\beta) = \nabla_A$ iff $\alpha \lor \beta = \nabla_A$;
\item $\rho([\alpha,\alpha]^n)=\rho(\alpha)$, for all integers $n \geq 0$.
\end{list}
\end{lema}

Recall that for an arbitrary family $(\alpha_i)_{i\in I}$ of congruences in $A$, the following equality holds: $\rho(\displaystyle \bigvee_{i \in I}\alpha_i)=\rho(\bigvee_{i \in I} \rho(\alpha_i))$. Then one can introduce the arbitrary joins in $RCon(A)$: if $(\alpha_i)_{i\in I} \subseteq RCon(A)$ then we denote $\displaystyle \bigvee_{i \in I}^{\cdot} \alpha_i=\rho(\bigvee_{i \in I}\alpha_i)$. Thus it is easy to prove that $(RCon(A), \displaystyle \bigvee^{\cdot}, \cap, \rho(\Delta_A), \nabla_A)$ is a frame (see \cite{Johnstone} as a basic text for the frame theory).

\begin{propozitie}\cite{Agliano}
Assume that $K(A)$ is closed under the commutator operation $[\cdot,\cdot]$. For any congruence $\theta$ of $A$ the following equality holds:

$\rho(\theta) = \bigvee\{\alpha\in K(A)|[\alpha,\alpha]^n\subseteq \theta$, for some  $n\geq 0\}$.
\end{propozitie}

In particular, we have $\rho(\Delta_A) = \bigvee\{\alpha\in K(A)|[\alpha,\alpha]^n = \Delta_A$, for some  $n\geq 0\}$.

Then the algebra $A$ is semiprime if and only if for any $\alpha\in K(A)$ and for any integer $n\geq 0$, $[\alpha,\alpha]^n = \Delta_A$ implies $\alpha = \Delta_A$.
.
Let $u:A\rightarrow B$ be an arbitrary morphism in $\mathcal{V}$ and $u^*:Con(B)\rightarrow Con(A)$, $u^{\bullet}:Con(A)\rightarrow Con(B)$ are the maps defined by $u^*(\beta) = u^{-1}(\beta)$ and $u^{\bullet}(\alpha) = Cg_B(f(\alpha))$, for all $\alpha\in Con(A)$ and $\beta\in Con(B)$. Thus $u^{\bullet}$ is the left adjoint of $u^*$: for all $\alpha\in Con(A)$, $\beta\in Con(B)$, we have $u^{\bullet}(\alpha)\subseteq \beta$ iff  $\alpha\subseteq u^*(\beta)$.

For any $\theta \in Con(A)$ we denote $V_A(\theta) = V(\theta) = \{\phi \in Spec(A)|\theta\subseteq \phi\}$ and $D_A(\theta) = D(\theta) = Spec(A)- V(\theta)$. If $\alpha, \beta \in Con(A)$ then $D(\alpha)\cap D(\beta) = D([\alpha,\beta])$ and $V(\alpha)\cup V(\beta) = V([\alpha,\beta])$. For any family of congruences $(\theta_i)_{i\in I}$ we have $\bigcup_{i\in I}D(\theta_i) = D(\bigvee_{i\in I}\theta_i)$ and $\bigcap_{i\in I}V(\theta_i) = V(\bigvee_{i\in I}\theta_i)$. Thus $Spec(A)$ becomes a topological space whose open sets are $D(\theta),\theta\in Con(A)$. We remark that this topology extends the Zariski topology (defined on the prime spectra of commutative rings) and the Stone topology (defined on the prime spectra of bounded distributive lattices). The properties of $Spec(A)$ were intensively studied by Agliano in \cite{Agliano} (for this reason we shall call $Spec(A)$ the Agliano spectrum of the algebra $A$).

 We mention that the family $(D(\alpha))_{\alpha\in K(A)}$ is a basis of open sets for the topology of $Spec(A)$. We remark that the set $Max(A)$ of maximal congruences of $A$ can be considered as a subspace of $Spec(A)$.

\section{Reticulation of a universal algebra}

\hspace{0.5cm} The reticulation of a ring $R$ is a bounded distributive lattice $L(R)$ whose prime spectrum (with the Stone topology \cite{BalbesDwinger}) is homeomorphic with the prime spectrum of $R$ (with the Zariski topology \cite{Atiyah}). This notion was generalized in \cite{GM2} to a universal algebra framework, then it was used to study remarkable classes of universal algebras \cite{GG}. In this section we shall remind the principal properties of the reticulation of a universal algebra (cf. \cite{GM2}).

Let us fix a semidegenerate congruence  modular variety $\mathcal{V}$ and $A$ an algebra of $\mathcal{V}$ such that the set $K(A)$ of compact congruences of $A$ is closed under the commutator operation. Consider  the following equivalence relation on $Con(A)$: for all $\alpha, \beta\in Con(A)$, $\alpha\equiv \beta$ if and only if $\rho(\alpha) = \rho(\beta)$. Let $\hat \alpha$ be the equivalence class of $\alpha\in Con(A)$ and $0 = \hat{\Delta_A}, 1 = \hat{\nabla_A}$. Then $\equiv$ is a congruence of the lattice $Con(A)$ so the quotient set $L(A)$ = $K(A)/{\equiv}$ is a bounded distributive lattice, named the reticulation of the algebra $A$ (see \cite{GM2}). We shall denote by $\lambda_A:K(A)\rightarrow L(A)$ the function defined by $\lambda_A(\alpha) = \hat{\alpha}$, for all $\alpha\in K(A)$.

We remark that for all $\alpha,\beta\in K(A)$ we have $\lambda_A(\alpha) = \lambda_A(\beta)$ if and only if $\rho(\alpha) = \rho(\beta)$.

\begin{lema}\cite{GM2}
 For all congruences $\alpha, \beta \in K(A)$ the following hold:
\usecounter{nr}
\begin{list}{(\arabic{nr})}{\usecounter{nr}}
\item $\lambda_A(\alpha \lor \beta) = \lambda_A(\alpha)\lor \lambda_A(\beta)$;
\item $\lambda_A(\alpha \cap \beta)$ = $\lambda_A([\alpha,\beta])$ = $\lambda_A(\alpha)\land \lambda_A(\beta)$;
\item $\lambda_A(\alpha) = 1$ iff $\alpha = \nabla_A$;
\item $\lambda_A(\alpha) = 0$ iff $[\alpha,\alpha]^k = \Delta_A$, for some integer $k\geq 1$;
\item $\lambda_A([\alpha,\alpha]^k) = \lambda_A(\alpha)$, for all integers $k\geq 1$;
\item $\lambda_A(\alpha) = 0$ iff $\alpha\subseteq \rho(\Delta_A)$;
\item If $A$ is semiprime then $\lambda_A(\alpha) = 0$ iff $\alpha = \Delta_A$;
\item  $\lambda_A(\alpha)\leq \lambda_A(\beta)$ iff $\rho(\alpha)\subseteq\rho(\beta)$ iff $[\alpha,\alpha]^n \subseteq \beta$, for some integer $n\geq 1$.
\end{list}
\end{lema}

Let $L$ be a bounded distributive lattice and $Id(L)$ the set of its ideals. Then $Spec_{Id}(L)$ will denote the set of prime ideals in $L$ and $Max_{Id}(L)$ the set of maximal ideals in $L$. For any ideal $I$ of $L$ we denote $D_{Id}(I) = \{Q\in Spec_{Id}(L)|I\not\subseteq Q\}$ and $V_{Id}(I) = \{Q\in Spec_{Id}(L)|I\subseteq Q\}$. If $x\in L$ then we use the notation $D_{Id}(x) = D_{Id}((x])  = \{Q\in Spec_{Id}(L)|x\notin Q\}$ and $V_{Id}(x) = V_{Id}((x])  = \{Q\in Spec_{Id}(L)|x\in Q\}$, where $(x]$ is the principal ideal of $L$ generated by the set $\{x\}$. Recall from \cite{Johnstone} that the family $(D_{Id}(x))_{x\in L}$ is a basis of open sets for the Stone topology on $Spec_{Id}(L)$.

For all $\theta\in Con(A)$ and $I\in Id(L(A))$ we shall denote

$\theta^{\ast} = \{\lambda_A(\alpha)|\alpha\in K(A), \alpha\subseteq \theta \}$ and $I_{\ast} =\bigvee\{\alpha\in K(A)|\lambda_A(\alpha)\in I\}$.

Thus $\theta^{\ast}$ is an ideal of the lattice $L(A)$ and $I_{\ast}$ is a congruence of $A$. In this way one obtains two order - preserving functions $(\cdot)^{\ast}:Con(A)\rightarrow Id(L(A))$ and $(\cdot)_{\ast}:Id(L(A))\rightarrow Con(A)$.

The functions $(\cdot)^{\ast}$ and $(\cdot)_{\ast}$ will play an important role in proving the transfer properties of reticulation. The following four results constitute the first steps in obtaining transfer properties. They will be used many times in the proofs.

\begin{lema}\cite{GG}
The following assertions hold
\usecounter{nr}
\begin{list}{(\arabic{nr})}{\usecounter{nr}}
\item If $\theta,\chi\in Con(A)$ then $[\theta,\chi]^{\ast}$ = $(\theta\cap \chi)^{\ast}$ = $\theta^{\ast}\bigcap \chi^{\ast}$;
\item If $(\theta_i)_{i\in I}$ is a family of congruences of $A$ then $(\bigvee_{i\in I}\theta_i)^{\ast} = \bigvee_{i\in I}\theta_i^{\ast}$.

\end{list}
\end{lema}

\begin{lema}
\cite{GM2} For all $\theta\in Con(A)$, $\alpha \in K(A)$ and $I\in Id(L(A))$ the following hold:
\usecounter{nr}
\begin{list}{(\arabic{nr})}{\usecounter{nr}}
\item $\alpha\subseteq I_{\ast}$ iff $\lambda_A(\alpha)\in I$;
\item $(\theta^{\ast})_{\ast} = \rho(\theta)$ and $(I_{\ast})^{\ast} = I$;
\item $\theta^{\ast} = (\rho(\theta))^{\ast}$ and $\rho(I_{\ast}) = I_{\ast}$;
\item If $\theta\in Spec(A)$ then $(\theta^{\ast})_{\ast} = \theta$ and $\theta^{\ast}\in Spec_{Id}(L(A))$;
\item If $I\in Spec_{Id}(L(A))$ then $I_{\ast}\in Spec(A)$;
\item If $\theta\in Spec(A)$ then $\alpha\subseteq \theta$ if and only if $\lambda_A(\alpha)\in \theta^{\ast}$;
\item If $\alpha\in K(A)$ then $\alpha^{\ast}$ is exactly the principal ideal $(\lambda_A(\alpha)]$ of the lattice $L(A)$.
\end{list}
\end{lema}

According to the previous lemma one can consider the order - preserving functions $u:Spec(A)\rightarrow Spec_{Id}(L(A))$ and $v:Spec_{Id}L((A))\rightarrow Spec(A)$, defined by $u(\phi) = \phi^{\ast}$ and $v(P) = P_{\ast}$, for all $\phi\in Spec(A)$ and $P\in Spec_{Id}(L(A))$. For any $\theta\in Con(A)$ we have $u(V(\theta)) = V_{Id}(\theta^{\ast})$ (see the proof of Proposition 4.17 in \cite{GM2}). By Lemma 3.4(7), for each $\alpha\in K(A)$ we have $u(V(\alpha)) = V_{Id}(\alpha^{\ast}) = V_{Id}(\lambda_A(\alpha))$.

\begin{propozitie}\cite{GM2}
The two functions $u:Spec(A)\rightarrow Spec_{Id}(L(A))$ and $v:Spec_{Id}L((A))\rightarrow Spec(A)$ are homeomorphisms, inverse to one another.
\end{propozitie}

\begin{propozitie}\cite{GM2}
The two functions $(\cdot)^{\ast}|_{RCon(A)}:RCon(A)\rightarrow Id(L(A))$ and $(\cdot)_{\ast}:Id(L(A))\rightarrow RCon(A)$ are frame isomorphisms, inverse to one another.
\end{propozitie}

\begin{remarca} We know that the prime spectrum $Spec_{Id}(L(A)$ of the bounded distributive lattice $L(A)$ is a spectral space in the sense of \cite{Hochster}, \cite{Dickmann} (or coherent space in the terminology of \cite{Johnstone}). By Proposition 3.4 it follows that $Spec(A)$ is a spectral space (see \cite{Agliano} for  a detailed discussion on the topological properties of $Spec(A)$).
\end{remarca}

\begin{remarca} By applying Proposition 3.4 it follows that $Max(A)$ is homeomorphic to the space $Max_{Id}(L(A)$ of maximal ideal of the lattice $L(A)$. According to \cite{Johnstone}, p.66, $Max_{Id}(L(A))$ is a compact $T1$-space, therefore $Max(A)$ is also a compact $T1$-space.
\end{remarca}

\section{Algebras whose reticulation preserves the Boolean center}

\hspace{0.5cm} Let us consider an algebra $A$ in a semidegenerate congruence  modular variety $\mathcal{V}$ such that the set $K(A)$ of finitely generated congruences of $A$ is closed under the commutator operation.

Let us denote by $B(Con(A))$ the set of complemented elements in the bounded lattice $Con(A)$. Since $Con(A)$ has a canonical structure of integral bounded unital $l$- groupoid in the sense of Section 3 of \cite{Jipsen}, by applying Lemma 4 from \cite{Jipsen} or \cite{Galatos}, it follows that $B(Con(A))$ is a Boolean algebra in which $\alpha^{\perp}$ is the complement of a congruence $\alpha\in B(Con(A))$. Then $B(Con(A))$ is said to be the Boolean center of $Con(A)$ (or, shortly, the Boolean center of $A$). We shall without mention the basic properties of complemented elements (see eg. Section 6 of \cite{GM2}). For example, a congruence $\theta$ is complemented if and only if $\alpha\lor \alpha^{\perp} = \nabla_A$. For all $\theta,\vartheta\in Con(A)$ and $\alpha\in B(Con(A))$ we have $\theta\cap \alpha =[\theta,\alpha]$, $\alpha\rightarrow \theta  = \alpha^{\perp}\lor \theta$ and $(\theta\cap\vartheta)\lor \alpha = (\theta\lor \alpha)\cap (\vartheta\lor \alpha)$.

\begin{lema} \cite{GM2}
 For all congruences $\theta, \vartheta\in Con(A)$ the following hold:
\usecounter{nr}
\begin{list}{(\arabic{nr})}{\usecounter{nr}}
\item If $\theta\lor \vartheta = \nabla_A$ and $[\theta,\vartheta] = \Delta_A$ then $\theta,\vartheta\in B(Con(A))$;
\item For any integer $n\geq 1$, if $\theta\lor \vartheta = \nabla_A$ and $[[\theta,\theta]^n,[\vartheta,\vartheta]^n] = \Delta_A$ then $[\theta,\theta]^n,[\vartheta,\vartheta]^n\in B(Con(A))$;
\item $B(Con(A))\subseteq K(A)$.

\end{list}
\end{lema}

If $L$ is a bounded distributive lattice then $B(L)$ will denote the Boolean algebra of complemented elements of $L$. The Boolean algebra $B(L)$ is called the Boolean center of $L$.

\begin{lema} \cite{GM2}
If $\alpha\in B(Con(A))$ then $\lambda_A(\alpha)\in B(L(A))$.
\end{lema}

The previous lemma allows us to consider the following map:

 $\lambda_A|_{B(Con(A))}: B(Con(A))\rightarrow B(L(A))$.

\begin{lema} \cite{GM2}
The map $\lambda_A|_{B(Con(A))}: B(Con(A))\rightarrow B(L(A))$ is an injective Boolean morphism.
\end{lema}

The following proposition characterizes the algebras $A$ of $\mathcal{V}$ for which $\lambda_A|_{B(Con(A))}$ is a Boolean isomorphism.

\begin{propozitie}
 The following assertions are equivalent:
\usecounter{nr}
\begin{list}{(\arabic{nr})}{\usecounter{nr}}
\item The map $\lambda_A|_{B(Con(A))}: B(Con(A))\rightarrow B(L(A))$ is a Boolean isomorphism;
\item The map $\lambda_A|_{B(Con(A))}: B(Con(A))\rightarrow B(L(A))$ is surjective;
\item For any $\alpha\in K(A)$, if $\lambda_A(\alpha)\in B(L(A))$ then there exists an integer $n\geq0$ such that $[\alpha,\alpha]^n\in B(Con(A))$;
\item For any $\alpha\in K(A)$, $\lambda_A(\alpha)\in B(L(A))$ if and only if there exists an integer $n\geq0$ such that $[\alpha,\alpha]^n\in B(Con(A))$.
\end{list}
\end{propozitie}

\begin{proof}
$(1)\Leftrightarrow(2)$ This equivalence follows by using Lemma 4.3.

$(2)\Rightarrow(3)$ Let $\alpha$ be a compact congruence of $A$ such that $\lambda_A(\alpha)\in B(L(A))$. According to the hypothesis $(2)$, there exists $\beta\in B(Con(A))$ such that $\lambda_A(\alpha) = \lambda_A(\beta)$. By Lemma 3.1(8), there exists an integer $n\geq0$ such that $[\alpha,\alpha]^n\leq \beta\leq\alpha$. From $\beta\in B(Con(A))$ we get $\beta\lor\gamma = \nabla_A$ and $[\beta,\gamma] = \Delta_A$, for some $\gamma\in B(Con(A))$. Thus $\alpha\lor\gamma = \nabla_A$, hence $[\alpha,\alpha]^n\lor[\gamma,\gamma]^n = \nabla_A$ (cf. Lemma 2.4(3)). We remark that $[[\alpha,\alpha]^n,[\gamma,\gamma]^n] \subseteq[\alpha,\gamma] = \Delta_A$, $[[\alpha,\alpha]^n,[\gamma,\gamma]^n] = \Delta_A$. Applying Lemma 4.1(2) we obtain $[\alpha,\alpha]^n\in B(Con(A)))$.

$(3)\Rightarrow(4)$ Assume that there exists an integer $n\geq0$ such that $[\alpha,\alpha]^n\in B(Con((A))$. Then $\lambda_A(\alpha) = \lambda_A([\alpha,\alpha]^n)\in B(L(A))$ (cf. Lemmas 3.1(5) and 4.2).

$(4)\Rightarrow(2)$ Assume that $\alpha$ is a compact congruence such that $\lambda_A(\alpha)\in B(L(A))$, so there exists an integer $n\geq0$ such that $[\alpha,\alpha]^n\in B(Con(A))$. Since $\lambda_A(\alpha) = \lambda_A([\alpha,\alpha]^n)$ and $[\alpha,\alpha]^n\in B(Con(A))$ it follows that $\lambda_A|_{B(Con(A))}$ is surjective.
\end{proof}

\begin{definitie}
We say that the reticulation of $A$ preserves the Boolean center if the equivalent conditions from Proposition 4.4 fulfill.
\end{definitie}

Let us consider the following property:

$(\star)$ For all $\alpha,\beta \in K(A)$ and for any integer $n\geq1$ the exists an integer $m\geq 0$ such that $[[\alpha,\alpha]^m,[\beta,\beta]^m]]\subseteq [\alpha,\beta]^n$.

If the commutator operation $[\cdot,\cdot]$ is associative then it is obvious that the algebra $A$ verifies the condition $(\star)$.

\begin{remarca} \cite{GM2}
If the algebra $A$ verifies the property $(\star)$ or is semiprime then the reticulation of $A$ preserves the Boolean center.
\end{remarca}

\begin{lema}
If $\alpha\in B(Con(A))$ then $D_A(\alpha)$ is a clopen subset of $Spec(A)$.
\end{lema}

\begin{proof} If $\alpha\in B(Con(A))$ then $\alpha\lor\beta=\nabla_A$ and $[\alpha,\beta]=\Delta_A$. for some $\beta\in B(Con(A))\subseteq K(A)$, so $D_A(\alpha)\cup D_A(\beta) = D_A(\alpha\lor\beta) = D_A(\nabla_A)=Spec(A)$ and $D_A(\alpha)\cap D_A(\beta)=D_A(\alpha\,\beta) = D_A(\Delta_A)=\emptyset$.
 \end{proof}

If $X$ is a topological space then we denote by $Clop(X)$ the Boolean algebra of clopen subsets of $X$. According to the previous lemma one can consider the map $D_A|_{B(Con(A))}: B(Con(A))\rightarrow Clop(Spec(A))$.

\begin{lema}
$D_A|_{B(Con(A))}: B(Con(A))\rightarrow Clop(Spec(A))$ is an injective Boolean morphism.
\end{lema}

\begin{proof} It is clear that $D_A|_{B(Con(A))}$ preserves the Boolean operations. In order to prove the injectivity of $D_A|_{B(Con(A))}$, observe that for any $\alpha\in Con(A)$, $D_A(\alpha)= Spec(A)$ if and only if $\alpha=\nabla_A$ (cf. Proposition 3.5 of \cite{GM2}).
\end{proof}

\begin{teorema}
If the reticulation of $A$ preserves the Boolean center then the map $D_A|_{B(Con(A))}: B(Con(A))\rightarrow Clop(Spec(A))$ is a Boolean isomorphism.
\end{teorema}

\begin{proof}
According to Lemma 4.8, it suffices to check the surjectivity of $D_A|_{B(Con(A))}$. Let $U$ be a clopen subset of $Spec(A)$. Then there exist $\theta,\chi\in Con(A)$ such that $U=V_A(\theta)$, $V_A(\theta\lor\chi) = V_A(\theta)\cap  V_A(\chi)= \phi$ and $V_A([\theta,\chi])= V_A(\theta)\cup V_A(\chi)= Spec(A)$, so $\theta\lor\chi = \nabla_A$ and $[\theta,\chi]\subseteq \rho(\Delta_A)$. Since $\nabla_A$ is compact and $K(A)$ is closed under joins there exist $\alpha,\beta\in K(A)$ such that $\alpha\subseteq\theta$, $\beta\subseteq\chi$, $\alpha\lor\beta=\nabla_A$ and $[\alpha,\beta]\subseteq \rho(\Delta_A)$. By Lemma 3.1,(1) and (2) it follows that $\lambda_A(\alpha)\lor\lambda_A(\beta)=\lambda_A(\alpha\lor\beta) = \lambda_A(\nabla_A) = 1$ and $\lambda_A(\alpha)\land\lambda_A(\beta)=\lambda_A([\alpha,\beta])=0$, hence $\lambda_A(\alpha),\lambda_A(\beta)\in B(Con(A))$.

In accordance with the hypothesis that the reticulation of $A$ preserves the Boolean center, from $\lambda_A(\alpha),\lambda_A(\beta)\in B(Con(A))$ it follows that $[\alpha,\alpha]^n, [\beta,\beta]^n\in B(Con(A))$, for some integer $n\geq 0$. By applying Lemma 2.4(3) we have the equality $[\alpha,\alpha]^n\lor [\beta,\beta]^n = \nabla_A$, hence we obtain $V_A([\alpha,\alpha]^n)\cap V_A([\beta,\beta]^n) = V_A([\alpha,\alpha]^n)\lor[\beta,\beta]^n) = V_A(\nabla_A) = \phi$. We observe that we have the following inclusions  $[[\alpha,\alpha]^n,[\beta,\beta]^n]\subseteq [\alpha,\beta]\subseteq \rho(\Delta_A)$, therefore $Spec(A) = V_A(\rho(\Delta_A)\subseteq V_A([[\alpha,\alpha]^n,[\beta,\beta]^n])= V_A([\alpha,\alpha]^n)\cup V_A([\beta,\beta]^n)$, so $V_A([\alpha,\alpha]^n)\cup V_A([\beta,\beta]^n)=Spec(A)$.

From $[\alpha,\alpha]^n\subseteq\alpha\subseteq\theta$ and $[\beta,\beta]^n\subseteq\beta\subseteq\chi$ we get $V_A([\alpha,\alpha]^n)\subseteq V_A(\theta)$ and  $V_A([\beta,\beta]^n)\subseteq V_A(\chi)$, therefore $U=V_A(\theta)=V_A([\alpha,\alpha]^n)=D_A([\beta,\beta]^n)$ (because $[\beta,\beta]^n$ is the complement of $[\alpha,\alpha]^n)$. Then we get the surjectivity of $D_A|_{B(Con(A))}$.

\end{proof}

If it is not danger of confusion we shall write $D_A$ instead of $D_A|_{B(Con(A))}$.

\begin{remarca} The previous theorem is a generalization to universal algebra of a classical result in ring theory: the map $e\mapsto D_A(e)$ is an isomorphism between the Boolean algebra of idempotents of a commutative ring $R$ and the Boolean algebra of a clopen subsets of $Spec(R)$ (see \cite{Jong}, 00EE ).
\end{remarca}

The converse of Theorem 4.9 is also true.

\begin{teorema} If the map $D_A|_{B(Con(A))}: B(Con(A))\rightarrow Clop(Spec(A))$ is a Boolean isomorphism then the reticulation of $A$ preserves the Boolean center.
\end{teorema}

\begin{proof} Assume that $\alpha\in K(A)$ and $\lambda_A(\alpha)\in B(L(A))$, so there exists $\beta\in K(A)$ such that $\lambda_A(\alpha)\lor\lambda_A(\beta)=1$ and $\lambda_A(\alpha)\land\lambda_A(\beta)=0$. Thus $\lambda_A(\alpha\lor\beta)=1$ and $\lambda_A([\alpha,\beta])=0$, hence $\alpha\lor\beta=1$ and $[\alpha,\beta]\subseteq \rho(\Delta_A)$.

It follows that $D_A(\alpha)\cup D_A(\beta)=Spec(A)$ and $D_A(\alpha)\cap D_A(\beta)=\emptyset$, hence $D_A(\alpha)=V_A(\beta)$ is a clopen subset of $Spec(A)$. Since $D_A|_{B(Con(A))}$ is a bijection there exists $\gamma\in B(Con(A))$ such that $D_A(\alpha)=D_A(\gamma)$, so $\lambda_A(\alpha)=\lambda_A(\gamma)$. By using Lemma 3.1(8) we get $\gamma\subseteq [\alpha,\alpha]^n\subseteq\gamma$, for some integer $n\geq 0$, so $[\alpha,\alpha]^n=\gamma\in B(Con(A))$. By the condition $(3)$ of Proposition 4.4, the reticulation of $A$ preserves the Boolean center.

\end{proof}

\section{Congruence Boolean Lifting Property}

\hspace{0.5cm}The lifting idempotent property ($LIP$) was studied in \cite{a} in relationship with the clean and the exchange rings. Recently, similar lifting properties were studied for other algebraic structures: bounded distributive lattices \cite{Cheptea}, commutative residuated lattices \cite{GCM}, universal algebras \cite{GKM},etc.
For example, in \cite{GKM} we introduced the notion of Congruence Boolean Lifting Property ($CBLP$) for the semidegenerate congruence modular algebras (see \cite{Agliano}). $CBLP$ generalizes $LIP$, as well as all the Boolean lifting properties existing in literature.

In this section we continue the investigations of \cite{GKM} on congruences and algebras with $CBLP$. Our results on congruences with $CBLP$ generalise to some universal algebras the main theorems proven by A. Tarizadeh and P.K. Sharma in \cite{Tar4} for lifting idempotents modulo an ideal.

Let $A$ be an algebra in a variety $\mathcal{V}$, $\theta$ a congruence of $A$ and $p_{\theta}:A\rightarrow A/{\theta}$ the canonical surjective morphism associated with the congruence $\theta$. According to Section 2, one can consider the map $p_{\theta}^{\bullet}: Con(A)\rightarrow Con(A/{\theta})$ associated with $p_{\theta}$. By Remark 5.19 of \cite{GKM}, we have $p_{\theta}(\alpha)= (\alpha\lor \theta)/{\theta}$ for each $\alpha\in Con(A)$. In virtue of Corollary 5.17 of \cite{GKM}, the map $p_{\theta}^{\bullet}: Con(A)\rightarrow Con(A/{\theta})$ induces a Boolean morphism $B(p_{\theta}^{\bullet}): B(Con(A))\rightarrow B(Con(A/{\theta}))$.

\begin{definitie} \cite{GKM} We shall say that a congruence $\theta$ of the algebra $A$ fulfills the Congruence Boolean Lifting Property ($CBLP$) if the Boolean morphism $B(p_{\theta}^{\bullet}): B(Con(A))\rightarrow B(Con(A/{\theta}))$ is surjective. The algebra $A$ fulfills $CBLP$ if each congruence of $A$ has $CBLP$.
\end{definitie}

In other words, $\theta\in Con(A)$ has $CBLP$ if and only if for any $\beta\in B(Con(A/{\theta}))$ there exists $\alpha\in B(Con(A))$ such that $(\alpha\lor \theta)/{\theta}=\beta$.

We observe that a commutative ring $R$ has $CBLP$ iff $R$ has $LIP$ iff $R$ is a clean ring (the last equivalence is a Nicholson theorem from \cite{a}).

 Let $A$ be an algebra in a fixed semidegenerate congruence  modular variety $\mathcal{V}$ such that the set $K(A)$ of finitely generated congruences of $A$ is closed under the commutator operation and let $\theta$ be a congruence of $A$. We shall denote $[\theta)_A = \{\chi\in Con(A)|\theta\subseteq\chi\}$ and $X_{\theta}=V_A(\theta)=\{\phi\in Spec(A)|\theta\subseteq\phi\}$. By \cite{Burris} it is well-known that any congruence of $A/{\theta}$ has the form $\chi/{\theta}= \{(x/{\theta},y/{\theta})|(x,y)\in \chi\}$, for some congruence $\chi\in [\theta)_A$ and the map $s_{\theta}:Con(A/{\theta})\rightarrow [\theta)_A$, defined by $s_{\theta}(\chi/{\theta})=\chi$, is a lattice isomorphism. Thus for all $\chi,\varepsilon\in [\theta)_A$, $\chi/{\theta} =\chi/{\theta}$ if and only if $\chi=\varepsilon$. We also know that the map $t_{\theta}= s_{\theta}|_{Spec(A/{\theta})}:Spec(A/{\theta})\rightarrow X_{\theta}$ is a homeomorphism.

 We remark that $(D_A(\chi)\cap [\theta)_A)_{\chi\in [\theta)_A}$ is a basis of open sets for the topology of $X_{\theta}$ and the map $t_{\theta}$ = $s_{\theta}|_{Spec(A/{\theta})}:Spec(A/{\theta})\rightarrow X_{\theta}$ is a homeomorphism. It is clear from the Stone duality \cite{Burris} that the homeomorphism $t_{\theta}$ induces the Boolean isomorphism $t_{\theta}^{\ast}:Clop(Spec(A/{\theta}))\rightarrow Clop(X_{\theta})$, given by $t_{\theta}^{\ast}(U)= t_{\theta}[U]=\{t_{\theta}(\epsilon)|\epsilon\in U\}$.

 For the rest of the section we will suppose that for any  algebra $A$ of the variety $\mathcal{V}$, the reticulation of $A$ preserves the Boolean center.

 According to Theorem 4.9, one can consider the Boolean isomorphism

 $D_{A/{\theta}}: B(Con(A/{\theta}))\rightarrow Clop(Spec(A/{\theta}))$.

Thus $v_{\theta}=t_{\theta}^{\ast}\circ D_{A/{\theta}}:B(Con(A/{\theta}))\rightarrow Clop(X_{\theta})$ is a Boolean isomorphism.

\begin{lema}For any congruence $\chi$ of $A$ such that $\theta\subseteq\chi$ and $\chi/{\theta}\in B(Con(A/{\theta}))$ we have $v_{\theta}(\chi/{\theta})=D_A(\chi)\cap[\theta)_A$.
\end{lema}

\begin{proof} Assume that $\chi\in Con(A)$, $\theta\subseteq\chi$ and $\chi/{\theta}\in B(Con(A/{\theta}))$. Then the following equalities hold:

$v_{\theta}(\chi/{\theta})=(t_{\theta}^{\ast}\circ D_{A/{\theta}})(\chi/{\theta})=t_{\theta}^{\ast}(\{\phi/{\theta})|\phi\in Spec(A)\cap[\theta)_A,\chi\not\subseteq\phi\}$= $\{t_{\theta}(\phi/{\theta})|\phi\in D_A(\chi)\cap[\theta)_A\}=D_A(\chi)\cap[\theta)_A$.

\end{proof}

We remark that $X_{\theta}=X_{\rho(\theta)}$ so $Clop(X_{\theta})= Clop(X_{\rho(\theta)})$. Since $v_{\theta}$ and $v_{\rho(\theta)}$ are Boolean isomorphisms the following lemma holds:

\begin{lema} There exists a Boolean isomorphism $\omega:B(Con(A/{\rho(\theta)}))\rightarrow B(Con(A/{\theta}))$ such that the following diagram is commutative.

\begin{center}

\begin{picture}(150,70)

\put(-50,50){$B(Con(A/_{\rho(\theta)}))$}

\put(25,55){\vector(1,0){100}}

\put(70,60){$v_{\rho(\theta)}$}

\put(130,50){$Clop(X_{\rho(\theta)})$}

\put(5,45){\vector(0,-1){30}}

\put(-10,30){$\omega$}

\put(-50,0){$B(Con(A/_{\theta}))$}

\put(25,5){\vector(1,0){100}}

\put(70,10){$v_\theta$}

\put(130,0){$Clop(X_\theta$)}

\put(135,45){\vector(0,-1){30}}

\put(140,30){$id$}

\end{picture}

\end{center}

where $id$ is the identity map.

\end{lema}

\begin{teorema}
 For any congruence $\theta\in Con(A)$ the following are equivalent:
\usecounter{nr}
\begin{list}{(\arabic{nr})}{\usecounter{nr}}
\item $\theta$ has $CBLP$;
\item $\rho(\theta)$ has $CBLP$.
\end{list}
\end{teorema}

\begin{proof}
$(1)\Rightarrow(2)$ Assume that $\theta$ has $CBLP$. Let $\chi$ be a congruence of $A$ such that $\rho(\theta)\subseteq\chi$ and $\chi/{\rho(\theta)}\in B(Con(A/{\theta}))$, hence $\omega(\chi/{\rho(\theta)})=\varepsilon/{\theta}$, for some congruence $\varepsilon$ such that $\theta\subseteq\varepsilon$. According to the commutative diagram of Lemma 5.3 the following equalities hold:

$D_A(\chi)\cap [\rho(\theta))_A = v_{\rho(\theta)}(\chi/{\rho(\theta)})=v_{\theta}(\omega(\chi/{\rho(\theta)}))=v_{\theta}(\varepsilon/{\theta})=D_A(\varepsilon)\cap[\theta)_A$.

Since $\omega$ is a Boolean isomorphism and $\chi/{\rho(\theta})\in B(Con(A/{\rho(\theta)}))$ we have $\varepsilon/{\theta}=\omega(\chi/{\rho(\theta)})\in B(Con(A/{\theta}))$. By applying the hypothesis that $\theta$ has $CBLP$ there exists $\alpha\in B(Con(A))$ such that $\varepsilon/{\chi}=p_{\theta}^{\bullet}(\alpha)=(\alpha\lor\theta)/{\theta}$, therefore $\varepsilon=\alpha\lor\theta$.

We want to show that $p_{\rho(\theta)}^{\bullet}(\alpha) = \chi/{\rho(\theta)}$, i.e. $(\alpha\lor\rho(\theta))/{\rho(\theta)}=\chi/{\rho(\theta)}$. Firstly we shall prove the equality $D_A(\alpha\lor\rho(\theta))\cap[\rho(\theta))_A = D_A(\chi)\cap[\rho(\theta))_A$. Assume that $\phi$ is a prime congruence of $A$, so it easy to see that the following equivalence holds: $\alpha\lor \rho(\theta)\not\subseteq \phi$ and $\rho(\theta)\subseteq \phi$ if and only if $\alpha\lor\theta\not\subseteq \phi$ and $\theta\subseteq \phi$. Thus we have the following equality: $D_A(\alpha\lor \rho(\theta))\cap [\rho(\theta))_A = D_A(\alpha\lor \theta)\cap [\theta)_A)$.

On the other hand, $D_A(\alpha\lor \theta)\cap [\theta)_A)=  D_A(\varepsilon)\cap [\theta)_A)= D_A(\chi)\cap[\rho(\theta))_A$, so $D_A(\alpha\lor\rho(\theta))\cap[\rho(\theta))_A = D_A(\chi)\cap[\rho(\theta))_A$. This last equality can be written under the form $v_{\rho(\theta)}((\alpha\lor\rho(\theta))/{\rho(\theta)})=v_{\rho(\theta)}(\chi/{\rho(\theta)})$. Since $v_{\rho(\theta)}$ is a Boolean isomorphism and  $(\alpha\lor\rho(\theta))/{\rho(\theta)}, \chi/{\rho(\theta)}\in B(Con(A/{\rho(\theta)}))$  we get $(\alpha\lor\rho(\theta))/{\rho(\theta)}=\chi/{\rho(\theta)}$. Then $p_{\rho(\theta)}^{\bullet}(\alpha) = \chi/{\rho(\theta)}$, so we conclude that $\rho(\theta)$ has $CBLP$.

$(2)\Rightarrow(1)$ Assume now that $\rho(\theta)$ has $CBLP$. Let $\varepsilon$ be a congruence of $A$ such that $\theta\subseteq\varepsilon$ and $\varepsilon/{\theta}\in B(Con(A/{\theta}))$. Since $\omega$ is a Boolean isomorphism and $\varepsilon/{\theta}\in B(Con(A/{\theta}))$ there exists $\chi\in Con(A)$ such that $\rho(\theta)\subseteq\chi$ such that $\chi/{\rho(\theta)}\in B(Con(A/{\rho(\theta)}))$ and $\omega(\chi/{\rho(\theta)})=\varepsilon/{\theta}$. As in the proof of implication $(1)\Rightarrow(2)$, by applying Lemma 5.3 one gets the equality $D_A(\chi)\cap [\theta)_A= D_A(\varepsilon)\cap [\theta)_A$. Since $\rho(\theta)$ has $CBLP$ and $\chi/{\rho(\theta)}\in B(Con(A/{\rho(\theta)}))$ there exists $\alpha\in B(Con(A))$ such that $\chi/{\rho(\theta)}=(\alpha\lor\rho(\theta))/{\rho(\theta})$, hence $\chi=\alpha\lor\rho(\theta)$.

For any $\phi\in Spec(A)$ we have $\alpha\lor \theta\subseteq\phi$  if and only if $\alpha\lor\rho(\theta)\subseteq\phi$. Thus $D_A(\alpha\lor\theta)\cap[\theta)_A= D_A(\alpha\lor\rho(\theta))\cap [\theta)_A=D_A(\chi)\cap [\theta)_A=D_A(\varepsilon)\cap [\theta)_A$, so $v_{\theta}((\alpha\lor\theta)/{\theta})=D_A(\alpha\lor\theta)\cap[\theta)_A=D_A(\varepsilon)\cap [\theta)_A=v_{\theta}(\varepsilon/{\theta})$. Since $v_{\theta}$ is a Boolean isomorphism and $(\alpha\lor\theta)/{\theta},\varepsilon/{\theta}\in B(Con(A/{\theta}))$ it follows that $p_{\theta}^{\bullet}(\alpha)=(\alpha\lor\theta)/{\theta}=\varepsilon/{\theta}$. Then $p_{\theta}^{\bullet}$ is surjective, so $\theta$ has $CBLP$.

\end{proof}

The previous theorem generalizes Proposition 3.1 of \cite{Tar4}. We observe that the proof of Theorem 5.4 is based on Theorem 4.9, which is a topological result. In what follows we shall present a short purely algebraic of Theorem 5.4 by using a result of \cite{GKM} concerning the way in which the reticulation preserves the Boolean lifting properties.

Let $L$ be a bounded distributive lattice. Following \cite{Cheptea} we say that an ideal $I$ of $L$ has $Id-BLP$ if for any $y\in B(L/I)$ there exists $x\in B(L)$ such that $x/I = y$. The lattice $L$ has $Id-BLP$ if any ideal of $L$ has $Id-BLP$.

\begin{propozitie} Assume that $A$ is an algebra of $\mathcal{V}$ and $\theta$ is a congruence of $A$. Then $\theta$ has $CBLP$ if and only if the ideal $\theta^{\ast}$ of the lattice $L(A)$ has $Id-BLP$. The algebra $A$ has $CBLP$ if and only if the lattice $L(A)$ has $Id-BLP$.
\end{propozitie}

\begin{proof}
By hypothesis, for any  algebra $A$ of variety $\mathcal{V}$, the reticulation of $A$ preserves the Boolean center, so one can apply Lemma 5.25 of \cite{GKM}.
\end{proof}

In particular, a commutative ring $R$ is a clean ring if and only if the lattice $L(R)$ has $Id-BLP$.

{\it{Second proof of Theorem 5.4}}. According to Lemma 3.3(3) we have $\theta^{\ast}=(\rho(\theta))^{\ast}$. By applying Proposition 5.5 the following equivalences hold: the congruence $\theta$ has $CBLP$ iff the ideal $\theta^{\ast}$ of the lattice $L(A)$ has $Id-BLP$ iff the ideal $(\rho(\theta))^{\ast}$ of the lattice $L(A)$ has $Id-BLP$ iff the congruence $\rho(\theta)$ has $CBLP$.

\begin{corolar}
 Let $I$ be an ideal of the lattice $L(A)$. Then $I$ has $Id-BLP$ if and only if the congruence $I_{\ast}$ of $A$ has $CBLP$.
\end{corolar}

\begin{proof}
Recall from Lemma 3.3(2) that $I= (I_{\ast})^{\ast}$, so by applying Proposition 5.5 the following equivalences hold: $I$ has $Id-BLP$ iff $(I_{\ast})^{\ast}$ has $Id-BLP$ iff $I_{\ast}$ has $CBLP$.
\end{proof}

\begin{corolar}
 If $\theta$ and $\chi$ are two congruences of $A$ such that $\rho(\theta)=\rho(\chi)$, then $\theta$ has $CBLP$ if and only if $\chi$ has $CBLP$.
\end{corolar}

\begin{proof}
By using Theorem 5.4 we have the following equivalences: $\theta$ has $CBLP$ iff $\rho(\theta)$ has $CBLP$ iff $\rho(\chi)$ has $CBLP$ iff $\chi$ has $CBLP$.
\end{proof}

\begin{corolar}
If $\theta$ and $\chi$ are two congruences of $A$ such that $\lambda_A(\theta)=\lambda_A(\chi)$, then $\theta$ has $CBLP$ if and only if $\chi$ has $CBLP$.
\end{corolar}

\begin{corolar}
Any congruence $\theta$ of $A$ contained in $\rho(\Delta_A)$ has $CBLP$.
\end{corolar}

\begin{proof} If the congruence $\theta$ of $A$ is contained in $\rho(\Delta_A)$ then $\rho(\theta)=\rho(\Delta_A)$. Since $\Delta_A$ has $CBLP$, by applying Corollary 5.7 it follows that $\theta$ has $CBLP$.

\end{proof}

\begin{lema}
 For all congruences $\theta,\varepsilon\in Con(A)$ such $\theta\subseteq\varepsilon$ the following hold:
\usecounter{nr}
\begin{list}{(\arabic{nr})}{\usecounter{nr}}
\item In $A/{\theta}$ we have $(\varepsilon/{\theta})^{\perp}=(\varepsilon\rightarrow\theta)/{\theta}$;
\item $\varepsilon/{\theta}\in B(Con(A/{\theta}))$ if and only if $\varepsilon\lor(\varepsilon\rightarrow\theta)=\nabla_A$.
\end{list}
\end{lema}

\begin{proof}

$(1)$ The equality $(\varepsilon/{\theta})^{\perp}=(\varepsilon\rightarrow\theta)/{\theta}$ was proven in \cite{GM4}.

$(2)$ We know that $\varepsilon/{\theta}\in B(Con(A/{\theta})$ iff $\varepsilon/{\theta}\lor (\varepsilon/{\theta})^{\perp}=  \nabla_{A/{\theta}}$, hence by using $(1)$, it follows that the following equivalences hold: $\varepsilon/{\theta}\in B(Con(A/{\theta}))$ iff $\varepsilon/{\theta}\lor (\varepsilon\rightarrow\theta)/{\theta}= \nabla_{A/{\theta}}$ iff $(\varepsilon\lor(\varepsilon\rightarrow\theta))/{\theta}=\nabla_{A/{\theta}}$ iff $\varepsilon\lor(\varepsilon\rightarrow\theta)=\nabla_A$.

\end{proof}

\begin{remarca}Assume that $\theta,\varepsilon\in Con(A)$ and $\theta\subseteq\varepsilon$. By Lemma 5.10 and the commutator property $(2.3)$ the following equalities hold: $([\varepsilon,\varepsilon\rightarrow\theta])\lor\theta)/{\theta} = ([\varepsilon/{\theta},(\varepsilon\rightarrow\theta)/{\theta}] = [\varepsilon/{\theta},(\varepsilon/{\theta})^{\perp}]= \Delta_{A/{\theta}} = \theta/{\theta}$, hence $[\varepsilon,\varepsilon\rightarrow \theta]\lor \theta=\theta$, therefore $[\varepsilon,\varepsilon\rightarrow\theta]\subseteq\theta$.
\end{remarca}

\begin{teorema} Assume that $\theta,\chi$ are two congruences of $A$ such that $\theta\subseteq\chi$ and $Max(A)\cap [\theta)_A =Max(A)\cap [\chi)_A $. If $\chi$ has $CBLP$ then $\theta$ has $CBLP$.
\end{teorema}

\begin{proof} Assume that $\chi$ has $CBLP$. Let $\varepsilon$ be a congruence of $A$ such that $\theta\subseteq\varepsilon$ and $\varepsilon/{\theta}\in B(Con(A/{\theta}))$, hence $\varepsilon\lor(\varepsilon\rightarrow\theta)=\nabla_A$ (by Lemma 5.10(2)). We observe that $\theta\subseteq\chi$ implies $\varepsilon\rightarrow\theta\subseteq\varepsilon\rightarrow\chi$, so $\nabla_A=\varepsilon\lor(\varepsilon\rightarrow\theta)\subseteq\varepsilon\lor(\varepsilon\rightarrow\chi)$, hence $\varepsilon\lor(\varepsilon\rightarrow\chi)=\nabla_A$. We know that $\chi\subseteq\varepsilon\rightarrow\chi$, so the following hold:

$\varepsilon\lor\chi\lor((\varepsilon\lor\chi)\rightarrow\chi)= \varepsilon\lor\chi\lor((\varepsilon\rightarrow\chi)\cap(\chi\rightarrow\chi))=\varepsilon\lor\chi\lor(\varepsilon\rightarrow\chi)=\varepsilon\lor(\varepsilon\rightarrow\chi)=\nabla_A$.

Applying again by Lemma 5.10(2) we get $(\varepsilon\lor\chi)/{\chi}\in B(Con(A/{\chi}))$. Recall that $\chi$ has $CBLP$, so there exists $\alpha\in B(Con(A))$ such that $(\alpha\lor\chi)/{\chi}=(\varepsilon\lor\chi)/{\chi}$, hence $\alpha\lor\chi=\varepsilon\lor\chi$.

Now we shall prove that $V_A(\alpha\lor\theta)=V_A(\varepsilon)$. In order to verify the inclusion $V_A(\alpha\lor\theta)\subseteq V_A(\varepsilon)$, let us assume that $\phi\in V_A(\alpha\lor\theta)$, so $\phi\in Spec(A)$ and $\alpha\lor\theta\subseteq \phi$. Let us take a maximal congruence $\psi$ of $A$ such that $\phi\subseteq\psi$, so $\alpha\subseteq\psi$ and $\theta\subseteq\psi$. It follows that $\psi\in Max(A)\cap [\theta)_A =Max(A)\cap [\chi)_A $, therefore $\chi\subseteq\psi$. Then $\varepsilon\lor\chi=\alpha\lor\chi\subseteq\psi$, so $\varepsilon\subseteq\psi$. Assume by absurdum that $\varepsilon\not\subseteq\phi$, hence $\varepsilon\rightarrow\theta\subseteq\phi$ (because by Remark 5.11 we have $[\varepsilon,\varepsilon\rightarrow\theta]\subseteq\theta\subseteq\phi$). From $\varepsilon\subseteq\psi$ and $\varepsilon\rightarrow\theta\subseteq\phi\subseteq\psi$ we get $\nabla_A=\varepsilon\lor(\varepsilon\rightarrow\theta)\subseteq\psi$, contradicting that $\psi\in Max(A)$. It follows that $\varepsilon\subseteq\phi$, i.e. $\phi\in V_A(\varepsilon)$. Then the desired inclusion $V_A(\alpha\lor\theta)\subseteq V_A(\varepsilon)$ is proven.

In order to show that $V_A(\varepsilon)\subseteq V_A(\alpha\lor\theta)$, let us suppose that $\phi\in V_A(\varepsilon)$, i.e. $\phi\in Spec(A)$ and $\varepsilon\subseteq\phi$. Let $\psi$ a maximal congruence of $A$ such that $\phi\subseteq\psi$, so $\theta\subseteq\varepsilon\subseteq\phi\subseteq\psi$. It follows that $\psi\in Max(A)\cap [\theta)_A =Max(A)\cap [\chi)_A $, hence $\chi\subseteq\psi$. Thus we obtain $\alpha\lor\chi=\varepsilon\lor\chi\subseteq\psi$. Assume by absurdum that $\alpha\lor\theta\not\subseteq\phi$, so $\alpha\not\subseteq\phi$ (because $\theta\subseteq\varepsilon\subseteq\phi$). Since  $\phi\in Spec(A)$ and $\alpha\in B(Con(A))$, $\alpha\not\subseteq\phi$ implies $\neg\alpha\subseteq\phi\subseteq\psi$, so $\nabla_A= \alpha\lor\neg\alpha\subseteq\phi$, contradicting that $\psi\in Max(A)$. Thus $\alpha\lor\theta\subseteq\phi$, hence $\phi\in V_A(\alpha\lor\theta)$ and the inclusion $V_A(\varepsilon)\subseteq V_A(\alpha\lor\theta)$ is obtained.

From $V_A(\alpha\lor\theta)=V_A(\varepsilon)$ we get $D_A(\alpha\lor\theta)=D_A(\varepsilon)$, therefore

$v_{\theta}((\alpha\lor\theta)/{\theta})=  D_A(\alpha\lor\theta)\cap [\theta)_A =D_A(\varepsilon)\cap [\theta)_A = v_{\theta}(\varepsilon/{\theta})$.

Since $v_{\theta}$ is injective and $(\alpha\lor\theta)/{\theta},\varepsilon/{\theta}\in B(Con(A/{\theta}))$, we get $p_{\theta}^{\bullet}(\alpha)=(\alpha\lor\theta)/{\theta}=\varepsilon/{\theta}$. Then $p_{\theta}^{\bullet}$ is surjective, so $\theta$ has $CBLP$.

\end{proof}

\begin{corolar} Let $\theta$ be a congruence of $A$ such that $\theta\subseteq Rad(A)$. If $Rad(A)$ has $CBLP$ then $\theta$ has $CBLP$.

\end{corolar}

\begin{proof} If $\theta$ is a congruence of $A$ such that $\theta\subseteq Rad(A)$ then $Max(A)\cap [\theta)_A =Max(A)\cap [Rad(A))_A = Max(A)$. Then the corollary follows by applying Theorem 5.12.

\end{proof}

In general $Rad(A)=\bigcap Max(A)$ does not verify $CBLP$ (see Example 3.10 of \cite{Tar4}). In this section we characterize the algebras $A$ for which $Rad(A)$ has $CBLP$. When $A$ is a commutative ring we obtain as particular cases some results proved in Section 3 of \cite{Tar4}.

The following theorem shows an interest in itself, but we shall use it as a tool in the proofs of this section.

\begin{teorema}
 If $U$ is a subset of $Max(A)$ then the following are equivalent
\usecounter{nr}
\begin{list}{(\arabic{nr})}{\usecounter{nr}}
\item $U$ is a clopen subset of $Max(A)$;
\item There exist $\alpha,\beta\in K(A)$ such that $\alpha\lor\beta=\nabla_A$, $[\alpha,\beta]\subseteq Rad(A)$ and $U=Max(A)\cap D_A(\alpha)$.
\end{list}
\end{teorema}

\begin{proof} $(1)\Rightarrow(2)$ Assume that $U\in Clp(Max(A))$. We know from Remark 3.7 that $Max(A)$ is a compact set and $U$ is a closed subset of $Max(A)$, so $U$ is itself compact, therefore there exist $\alpha_1,\cdots,\alpha_n\in K(A)$ such that $U=\bigcup_{i=1}^n (Max(A)\cap D_A(\alpha_i))$. By a similar argument, there exist $\beta_1,\cdots,\beta_m\in K(A)$ such that $Max(A)-U=\bigcup_{j=1}^m (Max(A)\cap D_A(\beta_j))$.

For all $i=1,\cdots,n$ and $j=1,\cdots,m$ we have $Max(A)\cap D_A([\alpha_i,\beta_j])= Max(A)\cap D_A(\alpha_i)\cap D_A(\beta_j)=\emptyset$, hence $\phi\in Max(A)$ implies $\phi\not\in D_A([\alpha_i,\beta_j])$, i.e. $[\alpha_i,\beta_j]\subseteq \phi$. It follows that  $[\alpha_i,\beta_j]\subseteq Rad(A)$, for all $i=1,\cdots,n$ and $j=1,\cdots,m$. Denoting $\alpha= \bigvee_{i=1}^n\alpha_i$ and $\beta= \bigvee_{j=1}^m\beta_j$ we get $\alpha,\beta\in K(A)$,

$U=Max(A)\cap (\bigcup_{i=1}^n D_A(\alpha_i)) = Max(A)\cap D_A(\alpha)$,

$Max(A)=(Max(A)\cap D_A(\alpha))\cup (Max(A)\cap D_A(\beta))= Max(A)\cap (D_A(\alpha)\cup\ D_A(\beta)) = Max(A)\cap D_A(\alpha\lor \beta)$

and $[\alpha,\beta]= \bigvee\{[\alpha_i,\beta_j]|i=1,\cdots,n ; j=1,\cdots,m\}\subseteq Rad(A)$.

If $\alpha\lor \beta \neq \nabla_A$ then $\alpha\lor \beta \subseteq \phi$ for some $\phi\in Max(A)$, hence $\phi\not\in D_A(\alpha\lor\beta)$, contradicting $Max(A)= Max(A)\cap D_A(\alpha\lor \beta)$. Then  $\alpha\lor \beta = \nabla_A$, hence the property $(2)$ is verified.

 $(2)\Rightarrow(1)$ Assume that there exist $\alpha,\beta\in K(A)$ such that $\alpha\lor\beta=\nabla_A$, $[\alpha,\beta]\subseteq Rad(A)$ and $U=Max(A)\cap D_A(\alpha)$. In order to prove that $U$ is a clopen subset of $Max(A)$ it suffices to show that $Max(a)\cap D_A(\alpha)= Max(A)\cap V_A(\beta)$.

If $\phi\in Max(A)\cap D_A(\alpha)$ then $\alpha\not\subseteq \phi$, hence $\beta\subseteq \phi$ (because $\phi$ is a prime congruence) and $[\alpha,\beta]\subseteq Rad(A)\subseteq \phi$). Thus $\phi\in Max(A)\cap V_A(\beta)$, therefore we get that $Max(a)\cap D_A(\alpha)\subseteq Max(A)\cap V_A(\beta)$.

Conversely, assume that $\phi\in Max(A)\cap V_A(\beta)$, hence $\beta\subseteq\phi\in Max(A)$. If $\alpha\subseteq\phi$ then $\nabla_A=\alpha\lor\beta\subseteq \phi$, contradicting that $\phi$ is a maximal congruence. Then $\alpha\not\subseteq\phi$, so $\phi\in Max(A)\cap D_A(\alpha)$, obtaining the inclusion $Max(A)\cap V_A(\beta)\subseteq Max(A)\cap D_A(\alpha)$. Thus the equality $Max(A)\cap D_A(\alpha)= Max(A)\cap V_A(\beta)$ follows, so $U$ is a clopen subset of $Max(A)$.

\end{proof}

Let $\alpha$ be a congruence of $A$ that fulfills the properties $Rad(A)\subseteq \alpha$ and $\alpha/{Rad(A)}\in B(Con(A/{Rad(A)}))$. Then there exists $\beta\in Con(A)$ such that the following hold: $Rad(A)\subseteq\beta$, $(\alpha\lor\beta)/{Rad(A)}=\alpha/{Rad(A)}\lor \beta/{Rad(A)}=\nabla_{A/Rad(A)}$ and $([\alpha,\beta]\lor Rad(A))/{Rad(A)}=[\alpha/Rad(A), \beta/Rad(A)]= \Delta_{Rad(A)}$. Thus we obtain $\alpha\lor\beta=\nabla_A$ and $[\alpha,\beta]\subset Rad(A)$. By using Theorem 6.1 it results that $Max(A)\cap D_A(\alpha)$ is a clopen subset of $Max(A)$, so one can define a map $f:B(Con(A/{Rad(A)})\rightarrow Clop(Max(A))$ by setting $f(\alpha/{Rad(A)})=Max(A)\cap D_A(\alpha)$.

\begin{propozitie} The map $f:B(Con(A/{Rad(A)})\rightarrow Clop(Max(A))$ is a Boolean isomorphism.
\end{propozitie}

\begin{proof} Let $\alpha,\beta$ be two congruences of $A$ such that $Rad(A)\subseteq\alpha\cap\beta$ and $\alpha/{Rad(A)},\beta/{Rad(A)}\in B(Con(A/{Rad(A)}))$. Then the following equalities hold:

$f(\alpha/{Rad(A)}\lor \beta/{Rad(A)})= f((\alpha\lor \beta)/{Rad(A)})= Max(A)\cap D_A(\alpha\lor\beta)= (Max(A)\cap D_A(\alpha))\cup (Max(A)\cap D_A(\beta) = f(\alpha/{Rad(A)})\cup f(\beta/{Rad(A)}) $.

Similarly, we have $f(\alpha/{Rad(A)}\cap \beta/{Rad(A)})=f(\alpha/{Rad(A)})\cap f(\beta/{Rad(A)}) $, so $f$ is a Boolean morphism. If $Max(A)\cap D_A(\alpha)=\emptyset$ then $\alpha=Rad(A)$, so $\alpha/{Rad(A)}= \Delta_{A/{Rad(A)}}$. We conclude that $f$ is an injective map.

In order to establish the surjectivity of $f$ let us suppose that $U\in Clop(Max(A))$. By using Theorem 5.14, there exist $\alpha,\beta\in K(A)$ such that $\alpha\lor\beta=\nabla_A$, $[\alpha,\beta]\subseteq Rad(A)$ and $U=Max(A)\cap D_A(\alpha)$. Then the following equalities hold:

$(\alpha\lor Rad(A))/{Rad(A)}\lor (\beta\lor Rad(A))/{Rad(A)}=\nabla_{A/{Rad(A)}}$

$[(\alpha\lor Rad(A))/{Rad(A)}, (\beta\lor Rad(A))/{Rad(A)}]= ([\alpha,\beta]\lor Rad(A))/{Rad(A)}=\Delta_{A/{Rad(A)}}$.

Let us denote $\gamma=\alpha\lor Rad(A), \delta=\beta\lor Rad(A)$, hence, by using the previous equalities we obtain $\gamma/{Rad(A)},\delta/{Rad(A)}\in B(Con(A/{Rad(A)}))$ and $Max(A)\cap D_A(\gamma)= Max(A)\cap (D_A(\alpha)\cup D_A(Rad(A))= Max(A)\cap D_A(\alpha)$ (because $D_A(Rad(A)=\emptyset$). It follows that $f(\gamma/{Rad(A)}= U$, so $f$ is surjective.

\end{proof}

\begin{lema}
The Boolean morphism $B(p_{Rad(A)}^{\bullet}):B(Con(A))\rightarrow B(Con(A/{Rad(A)}))$ is injective.
\end{lema}

\begin{proof} Assume that $\alpha\in B(Con(A))$ and $p_{Rad(A)}^{\bullet}(\alpha)=\Delta_{A/{Rad(A)}}$, so we have $(\alpha\lor Rad(A))/{Rad(A)}=\Delta_{A/{Rad(A)}}$. It follows that $\alpha\subseteq Rad(A)$. If $\alpha\neq\Delta_A$ then $\neg \alpha \neq \nabla_A$, so $\neg \alpha \subseteq \phi$ for some $\phi\in Max(A)$. On the other hand we have $\alpha\subseteq Rad(A)\subseteq\phi$, so $\nabla_A=\alpha\lor\neg\alpha \subseteq \phi$, contradicting that $\phi\in Max(A)$. We get $\alpha=\Delta_A$, hence $B(p_{Rad(A)}^{\bullet})$ is injective.
\end{proof}

Let us consider the map $g:B(Con(A))\rightarrow Clop(Max(A))$, defined by $g(\alpha)=Max(A)\cap D_A(\alpha)$, for any $\alpha\in B(Con(A))$. It is easy to see that $g$ is a Boolean morphism.

\begin{teorema}
$Rad(A)$ has $CBLP$ if and only if $g:B(Con(A))\rightarrow Clop(Max(A))$ is a Boolean isomorphism.
\end{teorema}

\begin{proof} For any $\alpha\in B(Con(A))$ the following equalities hold:

$f(B(p_{Rad(A)}^{\bullet})(\alpha))= f((\alpha\lor Rad(A))/{Rad(A)})= Max(A)\cap D_A(\alpha\lor Rad(A))= Max(A)\cap D_A(\alpha)=g(\alpha)$.

Then the following diagram is commutative in the category of Boolean algebras:

\begin{center}
\begin{picture}(150,70)
\put(-5,45){\tiny $B(Con(A))$}
\put(55,55){\tiny $B(p_{Rad(A)})$}
\put(20,50){ \vector(1,0){80}}
\put(105,45){\tiny $B(Con(A/\theta))$}
\put(10,40){\vector(2,-1){45}}
\put(20,20){\scriptsize$g$}
\put(65,20){\vector(2,1){45}}
\put(85,20){\scriptsize$f$}
\put(50,5){\scriptsize$Max(A)$}
\end{picture}
\end{center}

According to Proposition 5.15 $f$ is a Boolean isomorphism, hence, by using Lemma 5.16 and the previous commutative diagram, it follows that $g$ is an injective Boolean morphism. Thus we have the following equivalences: $Rad(A)$ has $CBLP$ iff $p_{Rad(A)}^{\bullet}$ is surjective iff $p_{Rad(A)}^{\bullet}$ is a Boolean isomorphism iff  $g$ is a Boolean isomorphism.

\end{proof}

\begin{remarca} If we apply Theorem 5.17 whenever $A$ is a commutative ring then we obtain Corollary 3.13 of \cite{Tar4}.

\end{remarca}

Recall from \cite{GM2} that the algebra $A$ is hyperarchimedean if for all $\alpha\in Con(A)$ there exists $n\geq 1$ such that $[\alpha,\alpha]^n\in B(Con(A))$.

\begin{propozitie} Any congruence $\theta$ of a hyperarchimedean algebra $A$ has $CBLP$.
\end{propozitie}

\begin{proof} Let $\theta$ be an arbitrary congruence of $A$. From Theorem 7.10 of \cite{GM2} we know that $A$ is hyperarchimedean if and only if the reticulation $L(A)$ of $A$ is a Boolean algebra. It is obvious that any ideal of a Boolean ring is a lifting ideal, so $\theta^{\ast}$ ia a lifting ideal of $L(A)=B(L(A))$ (i.e. $\theta^{\ast}$ has $Id-BLP$). Then, by applying Proposition 5.4, it follows that the congruence $\theta$ has $CBLP$.

\end{proof}

\section{Caracterization of congruences with $CBLP$}

\hspace{0.5cm} In this section we shall prove a characterization theorem for congruences of algebras that have $CBLP$. As a particular case one obtains the main part of Theorem 3.14 of \cite{Tar4}, a result that characterizes the lifting ideals in commutative rings. Before proving this characterization theorem we need some preliminary definitions and results.

Let us consider a semidegenerate congruence  modular variety $\mathcal{V}$ such that for any algebra $A\in \mathcal{V}$, the set $K(A)$ of finitely generated congruences of $A$ is closed under the commutator operation.

Let $R$ be a commutative ring and $B(R)$ the Boolean algebra of its idempotents. For an ideal $I$ of $R$ denote by $I^{\diamond}$ the ideal of $R$ generated by $I\cap B(R)$. If $L$ is a bounded distributive lattice and $I$ an ideal of $L$ then $I^{\diamond}$ will be the ideal of $L$ generated by $I\cap B(L)$ ($I^{\diamond}=[I\cap B(L))$. A similar construction can be done for any algebra $A\in \mathcal{V}$: if $\theta$ is a congruence of $A$ then $\theta^{\diamond}=\bigvee\{\alpha\in B(Con(A))|\alpha\subseteq\theta\}$.

For the rest of the section we will suppose that for any  algebra $A$ of the variety $\mathcal{V}$, the reticulation of $A$ preserves the Boolean center. We fix an algebra $A\in \mathcal{V}$.

The following proposition emphasizes the way in which the reticulation preserves the previous construction $\theta^{\diamond}$.

\begin{propozitie}
If $\theta\in Con(A)$ then $\theta^{{\ast}{\diamond}}=\theta^{{\diamond}{\ast}}$.
\end{propozitie}

\begin{proof} Firstly we shall prove that $\theta^{{\diamond}{\ast}}\subseteq \theta^{{\ast}{\diamond}}$. Assume that $x\in \theta^{{\diamond}{\ast}}$, so $x=\lambda_A(\alpha)$, for some $\alpha\in K(A)$ such that $\alpha\subseteq \theta^{\diamond}$. Since $\alpha$ is a compact congruence and $\theta^{\diamond}=\bigvee\{\alpha\in B(Con(A))|\alpha\subseteq\theta\}$ we get $\alpha\subseteq\gamma\subseteq\theta$, for some $\gamma\in B(Con(A))$. Thus $\lambda_A(\alpha)\leq\lambda_A(\gamma)$ and $\lambda_A(\gamma)\in B(L(A))\cap \theta^{\ast}$, hence it follows that $x=\lambda_A(\alpha)\in [B(L(A))\cap \theta^{\ast})=\theta^{{\ast}{\diamond}}$.

In order to establish the converse inclusion $\theta^{{\ast}{\diamond}}\subseteq\theta^{{\diamond}{\ast}}$ it suffices to check that $B(L(A))\cap \theta^{\ast}\subseteq\theta^{{\diamond}{\ast}}$. Assume that $x\in B(L(A))\cap \theta^{\ast}$, so $x=\lambda_A(\alpha)$ for some compact congruence $\alpha\subseteq\theta$. Since  the reticulation of $A$ preserves the Boolean center, $\lambda_A(\alpha\in B(L(A))$ implies $[\alpha,\alpha]^n\in B(Con(A))$, for some integer $n\geq 0$ (see the condition $(3)$ of Proposition 4.4). We observe that $[\alpha,\alpha]^n\leq\theta$ and $[\alpha,\alpha]^n\in B(Con(A))$ implies that $[\alpha,\alpha]^n\subseteq \theta^{\diamond}$, so $x=\lambda_A([\alpha,\alpha]^n)\in \theta^{{\diamond}{\ast}}$.
\end{proof}

Let us remind the following result proved by Hochster in \cite{Hochster}: for any bounded distributive lattice $L$ there exists a commutative ring $R$ whose reticulation $L(R)$ is isomorphic with $L$. By using the Hochster theorem, it follows that for the algebra $A$ there exists a commutative ring $R$ whose reticulation $L(R)$ is isomorphic to the lattice $L(A)$ (we can identify the isomorphic lattices $L(R)$ and $L(A)$). The reticulation properties presented in Section 3 provide a strong connection between the congruences of $A$ and the ideals of the ring $R$, allowing us to export some results from rings to algebras. We will illustrate this thesis in various proofs of this section

Let us consider $\theta\in Con(A)$ and $\phi\in Max(A)$, so $\theta^{\ast}$ is an ideal of the lattice $L(A)$ and $\phi^{\ast}$ is maximal ideal of $L(A)$. By using twice Proposition 3.4, there exist an ideal of $R$ and a maximal ideal $M$ of $R$ such that $I^{\ast}=\theta^{\ast}$ and $M^{\ast}=\phi^{\ast}$ (we identify the ring ideals $I$ and $M$ with their ring congruences). The cardinal number of a set $\Omega$ will be denoted by $|\Omega|$.

\begin{lema}
$|B(Con(A/{(\theta\lor\phi^{\diamond})}))|=|B(Con(R/{(I\lor M^{\diamond})}))|$.
\end{lema}

\begin{proof} By hypothesis, the reticulation of any algebra in $\mathcal{V}$ preserves the Boolean center. In particular, the reticulation of $A/{(\theta\lor\phi^{\diamond})}$ preserves the Boolean center. Therefore by using Theorem 4.4(3), Lemma 3.2(2) and Proposition 7.6 of \cite{GM2} one gets the following Boolean isomorphisms:

$B(Con(A/{(\theta\lor\phi^{\diamond})}))\simeq B(L(A/({\theta\lor\phi^{\diamond})}))\simeq B(L(A)/{(\theta\lor\phi^{\diamond})^{\ast}})\simeq $

$B(L(A)/{(\theta^{\ast}\lor\phi^{\diamond\ast}}))$

 and

 $B(Con(R/{(I\lor M^{\diamond})}))\simeq B(L(R/{(I\lor M^{\diamond})}))\simeq B(L(R)/{(I\lor M^{\diamond})^{\ast}})$

 $\simeq B(L(R)/{(I^{\ast}\lor M^{\diamond\ast}}))$.

 According to Proposition 6.1 we have $\phi^{{\diamond}{\ast}}=\phi^{{\ast}{\diamond}}=M^{{\ast}{\diamond}}=M^{{\diamond}{\ast}}$, hence the Boolean algebras $B(Con(A/{(\theta\lor\phi^{\diamond})}))$ and $B(Con(R/{(I\lor M^{\diamond})}))$ are isomorphic, so they have the same cardinal number.

\end{proof}

Recall from \cite{Atiyah} that two ideals $I,J$ of a commutative ring $R$ are said to be coprime if $I\lor J = I+J = R$. Similarly, two congruences $\theta,\chi$ of the algebra $A$ are coprime if $\theta\lor\chi=\nabla_A$. By Lemma 2.3(1), if the congruences $\theta,\chi$ are coprime then $[\theta,\chi]=\theta\cap\chi$.

\begin{teorema}
For any congruence $\theta$ of the algebra $A$ the following are equivalent:
\usecounter{nr}
\begin{list}{(\arabic{nr})}{\usecounter{nr}}
\item $\theta$ has $CBLP$;
\item For all coprime congruences $\phi,\psi$ of the algebra $A$ such that $[\phi,\psi]\subseteq \theta$ there exists $\alpha\in B(Con(A))$ such that $\alpha\subseteq\theta\lor\phi$ and $\neg\alpha\subseteq\theta\lor\psi$;
\item For all coprime congruences $\phi,\psi$ of the algebra $A$ such that $[\phi,\psi]= \theta$ there exists $\alpha\in B(Con(A))$ such that $\alpha\subseteq\theta\lor\phi$ and $\neg\alpha\subseteq\theta\lor\psi$;
\item If $\phi$ is a maximal congruence of $A$ then $|B(Con(A/{(\theta\lor\phi^{\diamond})}))|\leq 2$.
\end{list}
\end{teorema}

\begin{proof} $(1)\Rightarrow(2)$ Assume that $\phi,\psi$ are coprime congruences of the algebra $A$ such that $[\phi,\psi]\subseteq \theta$, so $(\phi\lor\theta)/{\theta}\lor (\psi\lor\theta)/{\theta}=\nabla_{A/{\theta}}$ and $[(\phi\lor\theta)/{\theta},(\psi\lor\theta)/{\theta}]=([\phi,\psi]\lor\theta)/{\theta}=\theta/{\theta}=\Delta_{A/{\theta}}$, therefore $(\phi\lor\theta)/{\theta}, (\psi\lor\theta)/{\theta}\in B(Con(A/{\theta}))$ and $\neg((\phi\lor\theta)/{\theta})= (\psi\lor\theta)/{\theta}$. By hypothesis, $\theta$ has $CBLP$, i.e. the map $B(p_{\theta}^{\bullet}):B(Con(A))\rightarrow B(Con(A/{\theta})))$ is surjective, so there exist $\alpha\in B(Con(A))$ such that $(\phi\lor\theta)/{\theta}= p_{\theta}^{\bullet}(\alpha)=(\alpha\lor\theta)/{\theta}$. It follows that $\phi\lor\theta=\alpha\lor\theta$.

On the other hand, $B(p_{\theta}^{\bullet})$ is a Boolean morphism and $\neg\alpha\in B(Con(A))$, hence the following equalities hold:

$(\neg\alpha\lor\theta)/{\theta}=p_{\theta}^{\bullet}(\neg\alpha)=\neg p_{\theta}^{\bullet}(\alpha)=\neg((\alpha\lor\theta)/{\theta})=\neg((\phi\lor\theta)/{\theta})=(\psi\lor\theta)/{\theta}$,

hence $\neg\alpha\lor\theta=\psi\lor\theta$, therefore $\neg\alpha\subseteq\theta\lor\psi$.

$(2)\Rightarrow(3)$ Obviously.

$(3)\Rightarrow(1)$ In order to show that $B(p_{\theta}^{\bullet})$ is a surjective map, let us consider $\chi\in Con(A)$ such that $\theta\subseteq\chi$ and $\chi/{\theta}\in B(Con(A/{\theta}))$. Thus there exists $\varepsilon\in Con(A)$ such that $\theta\subseteq\varepsilon$ and $\neg( \chi/{\theta})=\varepsilon/{\theta}$. It follows that $(\chi\lor\varepsilon)/{\theta}=\chi/{\theta}\lor\varepsilon/{\theta}=\nabla_{A/{\theta}}$, so $\chi\lor\varepsilon=\nabla_A$, i.e. $\chi$ and $\varepsilon$ are coprime.

From $(\chi\cap\varepsilon)/{\theta}=(\chi/{\theta}\cap\varepsilon)/{\theta}=\Delta_{A/{\theta}}$ we get $\chi\cap\varepsilon=\theta$. Since $\chi$ and $\varepsilon$ are coprime we have $[\chi,\varepsilon]=\chi\cap\varepsilon=\theta$. Then one can apply the hypothesis $(3)$, so there exists $\alpha\in B(Con(A))$ such that $\alpha\subseteq\chi\lor\theta=\chi$ and $\neg\alpha\subseteq\varepsilon\lor\theta=\varepsilon$.

From $\alpha\subseteq\chi$ we get $p_{\theta}^{\bullet}(\alpha)=(\alpha\lor\theta)/{\theta}\subseteq\chi/{\theta}$. On the other hand, the following implications hold: $\neg\alpha\subseteq\varepsilon\Rightarrow  p_{\theta}^{\bullet}(\neg\alpha)\subseteq p_{\theta}^{\bullet}(\varepsilon)\Rightarrow\neg p_{\theta}^{\bullet}(\alpha)\subseteq\varepsilon/{\theta}\Rightarrow \neg \varepsilon/{\theta}\subseteq p_{\theta}^{\bullet}(\alpha)\Rightarrow \chi/{\alpha}\subseteq p_{\theta}^{\bullet}(\alpha)$.

It follows that $ p_{\theta}^{\bullet}(\alpha)=\chi/{\theta}$ and $\alpha\in B(Con(A))$, so $ p_{\theta}^{\bullet}$ is surjective. Then $\theta$ has $CBLP$.

$(1)\Leftrightarrow(4)$ By using the Hochster theorem \cite{Hochster}, one can find a commutative ring $R$ such that the reticulation $L(A)$ of the algebra $A$ and the reticulation $L(R)$ of the ring $R$ are isomorphic. As usual, we can assume that $L(A)$ and $L(R)$ coincide.

We know that $\theta^{\ast}$ is an ideal of $L(A)$ so one can take an ideal $I$ of the ring $R$ such that $\theta^{\ast}=I^{\ast}$ (we identify the ideals of $R$ with their congruences). According to Remark 3.7, for any maximal congruence $\phi$ of $A$ we can find a maximal ideal $M$ of $R$ such that $\phi^{\ast}=M^{\ast}$.

By two applications of Proposition 5.5 it follows that $\theta$ has $CBLP$ if and only if $I$ is a lifting ideal of $R$. Therefore by using Theorem 3.14 of \cite{Tar4} and Lemma 6.2, the following properties are equivalent:

$\bullet$ $\theta$ has $CBLP$;

$\bullet$ $I$ is a lifting ideal of $R$;

$\bullet$ If $M$ is a maximal ideal of $R$ then $R/({I\lor M^{\diamond})}$ has no nontrivial idempotent;

$\bullet$ If $\phi$ is a maximal congruence of $A$ then $|B(Con(A/{(\theta\lor\phi^{\diamond})}))|\leq 2$.

\end{proof}

Recall from \cite{Aghajani} that an ideal of a commutative ring $R$ is said to be regular if $I=I^{\diamond}$; similarly, an ideal $J$ of a bounded distributive  $L$ is said to be regular if $J=J^{\diamond}$. In general, a congruence $\theta$ of the algebra $A$ is regular if $\theta=\theta^{\diamond}$.

\begin{lema}
Assume that $\theta\in Con(A)$ and $I\in Id(L(A))$.
\usecounter{nr}
\begin{list}{(\arabic{nr})}{\usecounter{nr}}
\item If $\theta$ is a regular congruence of $A$ then $\theta^{\ast}$ is a regular ideal in the lattice $L(A)$;
\item If $I$ is a regular ideal in the lattice $L(A)$ then there exists a regular congruence $\chi$ of $A$ such that $I=\chi^{\ast}$.
\end{list}
\end{lema}

\begin{proof}
$(1)$ See Proposition 6.19(1) of \cite{GG}.

$(2)$ $\chi=(I_{\ast})^{\diamond}$ is a regular congruence of $A$ and by using Proposition 6.1 and Lemma 3.3(2), the following equalities hold: $\chi^{\ast}=(I_{\ast})^{\diamond \ast} = ((I_{\ast})^{\ast})^{\diamond}=I^{\diamond}=I$.
\end{proof}

\begin{propozitie} Let $\theta, \chi$ be two congruences of $A$. If $\theta$ has $CBLP$ and $\chi$ is regular then $\theta\lor\chi$ has $CBLP$.
\end{propozitie}

\begin{proof} Let $R$ be a commutative ring such that the lattices $L(A)$ and $L(R)$ are identical (by the Hochster theorem). Consider an arbitrary maximal congruence $\phi$ of $A$. In order to prove that $\theta\lor\chi$ has $CBLP$ it suffices to establish the inequality $|B(Con(A/(\theta\lor\chi\lor\phi^{\diamond})))|\leq 2$ (according to Theorem 6.3).

Let $I$ be an ideal of $R$ and $M$ a maximal ideal of $R$ such that $I^{\ast}=\theta^{\ast}$ and $M^{\ast}=\phi^{\ast}$. By Lemma 6.4(1), $\chi{\ast}$ is a regular ideal of $L(A)=L(R)$, hence, by using Lemma 6.4(2) for $R$, we can find a regular ideal $J$ of $R$ such that $J^{\ast}=\chi^{\ast}$.

We observe that $\phi^{\diamond \ast}= M^{\diamond \ast}$ (by Proposition 6.1). Remind that the reticulation of any algebra in  $\mathcal{V}$ preserves the Boolean center, hence the reticulation of $A/(\theta\lor\chi\lor\phi^{\diamond})$ preserves the Boolean center. According to Proposition 4.4, Lemma 3.2(2) and Proposition 7.6 of \cite{GM2} we have the following Boolean isomorphisms:

$B(Con(A/(\theta\lor\chi\lor\phi^{\diamond})))\simeq B(L(A)/(\theta\lor\chi\lor\phi^{\diamond})^{\ast})\simeq B(L(A)/(\theta^{\ast}\lor\chi^{\ast}\lor \phi^{\diamond \ast}))\simeq B(L(R)/(I{\ast}\lor J^{\ast}\lor M^{\diamond \ast}))\simeq...\simeq B(Con(R/(I\lor J\lor M^{\diamond})))$,

therefore  $|B(Con(A/(\theta\lor\chi\lor\phi^{\diamond})))|=|B(Con(R/(I\lor J\lor M^{\diamond})))|$.

Since $\theta$ has $CBLP$ it follows that $I$ is a lifting ideal of $R$ (by two applications of Proposition 5.5) and $J$ is a regular ideal of $R$, therefore, by using Corollary 3.15 of \cite{Tar4} it follows that $I \lor J$ is a lifting ideal of $R$. Then $R/(I\lor J \lor M^{\diamond})$ has no nontrivial idempotent, i.e. $|B(Con(R/(I\lor J\lor M^{\diamond})))|\leq 2$, hence we get the desired inequality $|B(Con(A/(\theta\lor\chi\lor\phi^{\diamond})))|\leq 2$.
\end{proof}

\begin{corolar} Any regular congruence $\chi$ of $A$ has $CBLP$.
\end{corolar}

\begin{proof}  We know that $\Delta_A$ has $CBLP$. If we take $\theta=\Delta_A$ in the previous proposition it follows that $\chi= \Delta_A\lor\chi$  has $CBLP$.
\end{proof}

\begin{propozitie} Let $\theta,\chi$ be two non-coprime congruences of $A$ such that $\theta$ has $CBLP$ and $|B(Con(A/{\chi})|\leq 2$. Then $\theta\cap\chi$ has $CBLP$.
\end{propozitie}

\begin{proof} Let $R$ be a commutative ring such that $L(A)=L(R)$ and $I, J$ two ideals of $R$ such that $I^{\ast}=\theta^{\ast}$ and $J^{\ast}=\chi^{\ast}$. By two applications of Proposition 5.5 it follows that $I$ is a lifting ideal of $R$ and by using Proposition 4.4(1) and Proposition 7.6 of \cite{GM2}, one can prove that $|B(Con(A/{\chi}))|=|B(Con(R/J))| $, hence $|B(Con(R/J))|\leq 2 $, i.e $R/J$ has no nontrivial idempotents. Let  $\phi$ be a maximal congruence of $A$ such that $\theta\lor\chi\subseteq\phi$ and $M$ a maximal ideal of $R$ such that $M^{\ast}=\phi^{\ast}$  hence, by using Lemma 3.2(2), one gets

$(I\lor J)^{\ast}= I^{\ast}\lor J^{\ast} =  \theta^{\ast}\lor \chi^{\ast} = (\theta\lor \chi)^{\ast}\subseteq\phi^{\ast}= M^{\ast}$.

Applying Lemma 3.3(2) we obtain $I
\lor J=((I\lor J)^{\ast})_{\ast}\subseteq (M^{\ast})_{\ast}=M$, hence $I,J$ are non-coprime ideals of $R$. In accordance with Proposition 3.17 of \cite{Tar4}, $I\cap J$ is a lifting ideal of $R$. By Lemma 3.3(1) we have $(\theta\cap\chi)^{\ast}=\theta^{\ast}\cap\chi^{\ast}=I^{\ast}\cap J^{\ast} = (I\cap J)^{\ast}$. By two applications of Proposition 5.5 it follows that $\theta\cap\chi$ has $CBLP$.

\end{proof}

Recall from \cite{a} that a clean ring is a ring $R$ such that any element of $R$ is the sum of a unit and an idempotent. We know from \cite{a} that a commutative ring $R$ is clean if and only if any ideal of $R$ is a lifting ideal.

\begin{propozitie} Let $\theta$ be a congruence of $A$ such that $\theta\subseteq Rad(A)$. If $A/{\theta}$ has $CBLP$ then $A$ has $CBLP$.

\end{propozitie}

\begin{proof} Let $R$ be a commutative ring such that $L(A)= L(R)$ and $I$ an ideal of $R$ such that $I^{\ast}=\theta^{\ast}$. By using Lemma 3.3 and the bijective correspondence between $Max(A)$ and $Max(R)$ (cf. Remark 3.7) we can prove that $I\subseteq Rad(R)$.

According to Proposition 7.6 of \cite{GM2}, we have the following Boolean isomorphisms: $L(A/{\theta})\simeq L(A)/\theta^{\ast}=L(R)/I^{\ast}=L(R/I)$. By applying Proposition 5.5, the following implications hold:

$A/{\theta}$ has $CBLP$ $\Rightarrow$ $L(A/{\theta})$ has $Id-BLP$

 $\hspace{2.4cm}$                       $\Rightarrow$ $L(R/{\theta})$ has $Id-BLP$

$\hspace{2.4cm}$                         $\Rightarrow$ $R/I$ is a clean ring.

In accordance with Proposition 1.5 of \cite{a} or Corollary 6.6 of \cite{Tar4} it results that $R$ is a clean ring. Thus, a new application of Proposition 5.5 shows that the algebra $A$ has $CBLP$.

\end{proof}

Following \cite{GKM}, the algebra $A$ is said to be $B$-normal if for all coprime congruences $\chi,\varepsilon$ of $A$ there exist $\alpha, \beta\in B(Con(A))$ such that $\chi\lor\alpha=\varepsilon\lor\beta=\nabla_A$ and $[\alpha,\beta]=\Delta_A$. By using Theorem 6.3 one can give a short proof of the following theorem from \cite{GKM}.

\begin{teorema}
 The following are equivalent:
\usecounter{nr}
\begin{list}{(\arabic{nr})}{\usecounter{nr}}
\item The algebra $A$ has $CBLP$;
\item $A$ is a $B$-normal algebra.
\end{list}
\end{teorema}

\begin{proof} $(1)\Rightarrow(2)$ Let  $\chi,\varepsilon$ be two coprime congruences of $A$. By hypothesis, $\theta=[\chi,\varepsilon]$ has $CBLP$, so one can apply the condition $(3)$ of Theorem 6.3, so there exists $\alpha\in B(Con(A))$ such that $\alpha\subseteq\theta\lor\chi=[\chi,\varepsilon]\lor \chi=\chi$ and $\neg\alpha\subseteq\theta\lor\varepsilon=[\chi,\varepsilon]\lor \varepsilon=\varepsilon$. Denoting $\beta=\neg\alpha$ we have $\nabla_A=\alpha\lor \neg\alpha\subseteq\chi\lor\beta$ and $\nabla_A=\neg\alpha\lor\alpha\subseteq\varepsilon\lor\alpha$, so $\chi\lor\beta=\varepsilon\lor\alpha=\nabla_A$ and $[\alpha,\beta]= [\alpha,\neg\alpha]=\Delta_A$.

$(2)\Rightarrow(1)$ Assume that $\theta\in Con(A)$ and $\chi,\varepsilon$ are two coprime congruences of $A$ such that $\theta=[\chi,\varepsilon]$. Since $A$ is a $B$-normal algebra and $\chi,\varepsilon$ are two coprime we can find $\alpha,\beta\in B(Con(A))$ such that $\chi\lor\alpha=\varepsilon\lor\beta=\nabla_A$ and $[\alpha,\beta]=\Delta_A$. It is easy to see that $\neg\alpha\subseteq\chi$ and $\alpha\subseteq\neg\beta\subseteq\varepsilon$. It follows that $\alpha\subseteq\varepsilon=\varepsilon\lor[\chi,\varepsilon]=\varepsilon\lor\theta$ and $\neg\alpha\subseteq\chi=\chi\lor[\chi,\varepsilon]=\chi\lor\theta$, so $\theta$ has $CBLP$ (by applying Theorem 6.3).

\end{proof}

\begin{teorema} Let $\theta$ be a congruence of $A$ such that $\theta\subseteq Rad(A)$ and $\theta$ has $CBLP$. If $A/{\theta}$ is $B$-normal then $A$ is $B$-normal.
\end{teorema}

\begin{proof} By Theorem 6.9, the quotient algebra $A/{\theta}$ has $CBLP$. Let $R$ be a commutative ring and $I$ an ideal of $R$ such that $L(A)=L(R)$ and $\theta^{\ast}=I^{\ast}$. Assume that $M$ is a maximal ideal of $R$, so $M^{\ast}=\phi^{\ast}$, for some maximal congruence $\phi$ of $A$. Thus $\theta\subseteq Rad(A)\subseteq \phi$, hence $I^{\ast}= \theta^{\ast}\subseteq \phi^{\ast}=M^{\ast}$. By applying Lemma 3.3(2), we get $I=(I^{\ast})_{\ast}\subseteq (M^{\ast})_{\ast}= M$. We have proven that $I\subseteq Rad(R)$. On the other hand, by Proposition 7.6 of \cite{GM2} we get the following lattice isomorphisms: $L(A/{\theta})\simeq L(A)/\theta^{\ast} = L(R)/M^{\ast}=L(R/I)$. According to Proposition 5.5, it results that $R/I$ has $CBLP$, so it is a clean ring. By hypothesis, $\theta$ has $CBLP$, hence $I$ is a lifting ideal of $R$ (by two applications of Proposition 5.5).

We remark that $R$ and $I$ verifies the hypotheses of Corollary 6.6 of \cite{Tar4}, so $R$ is a clean ring. By two applications of Proposition 5.5 it follows that $A$ has $CBLP$. A new application of Theorem 6.9 shows that $A$ is $B$-normal.

\end{proof}

\section{Lifting orthogonal sets of Boolean congruences}

\hspace{0.5cm} In this section we shall prove a characterization theorem for congruences of algebras that have $CBLP$.

Let us consider a semidegenerate congruence  modular variety $\mathcal{V}$ such that for any algebra $A\in \mathcal{V}$, the set $K(A)$ of finitely generated congruences of $A$ is closed under the commutator operation.

Let $u:A\rightarrow A'$ be a morphism of algebras in the variety $\mathcal{V}$ and $\Omega'$ a subset of $Con(A')$. We say that the morphism $u$ lifts $\Omega'$ to a subset  $\Omega$ of $Con(A)$ if $\Omega'= \{u^{\bullet}(\alpha)|\alpha\in \Omega\}$. The morphism $u$ lifts the Boolean congruences if it lifts $B(Con(A'))$ to $B(Con(A))$. Of course, a congruence $\theta$ of an algebra $A\in \mathcal{V}$ has $CBLP$ if and only if $p_{\theta}:A\rightarrow A/{\theta}$ lifts the Boolean congruences.

Recall that a subset $E$ of a Boolean algebra $B$ is called an orthogonal set if $e\land f = 0$, for all distinct $e,f\in E$. Then an orthogonal set of Boolean congruences of an algebra $A$ is a subset $\Omega$ of $B(Con(A))$ such that $[\alpha,\beta]= \alpha\cap \beta = \Delta_A$, for all distinct $\alpha,\beta\in \Omega$.

\begin{teorema} If the morphism $u:A\rightarrow A'$ lifts the Boolean congruences and $\Omega'$ is a countable orthogonal set in $B(Con(A'))$ then there exists a countable orthogonal set $\Omega$ in $B(Con(A))$ such that $u$ lifts $\Omega'$ to $\Omega$.
\end{teorema}

\begin{proof} Let $\Omega'= \{\beta_1,...,\beta_n,...\}$ be a countable orthogonal subset of the Boolean algebra $B(Con(A'))$. We shall construct by induction a countable orthogonal subset $\Omega=\{\alpha_1,...,\alpha_n,...\}$ of the Boolean algebra $B(Con(A))$ such that the morphism $u$ lifts $\Omega'$ to $\Omega$.

Assume that $\{\alpha_1,...,\alpha_n\}$ is an orthogonal subset of $B(Con(A))$ such that $u^{\bullet}(\alpha_i) = \beta_i$, for all $i=1\cdots,n$. Since $u:A\rightarrow A'$ lifts the Boolean congruences there exists $\alpha\in B(Con(A))$ such that $u^{\bullet}(\alpha)= \beta_{n+1}$. Let us denote $\alpha_{n+1}=\alpha\land \neg(\bigvee_{i=1}^n\alpha_i)$, so $\{\alpha_1,...,\alpha_{n+1}\}$ is an orthogonal subset of $B(Con(A))$. Since $\{\beta_1,...,\beta_{n+1}\}$ is an orthogonal subset of $B(Con(A'))$ we have $\beta_{n+1}\leq\neg(\bigvee_{i=1}^n\beta_i)$, therefore, by taking into account that $u^{\bullet}|_{B(Con(A))}$ is a Boolean morphism, we get

$u^{\bullet}(\alpha_{n+1})= u^{\bullet}(\alpha)\land \neg(\bigvee_{i=1}^n u^{\bullet}(\alpha_i))=\beta_{n+1}\land\neg(\bigvee_{i=1}^n\beta_i) = \beta_{n+1}$.

\end{proof}

Recall that for all $\alpha,\beta\in B(Con(A))$ we denote $\alpha -\beta = \alpha\cap \neg\beta$.

\begin{lema}

\usecounter{nr}
\begin{list}{(\arabic{nr})}{\usecounter{nr}}
\item If $\alpha\in B(Con(A))$ then $\alpha\subseteq Rad(A)$ implies $\alpha=\Delta_A$;
\item If $\alpha,\beta\in B(Con(A))$ then $\alpha-\beta\subseteq Rad(A)$ implies $\alpha=\beta$.
\end{list}
\end{lema}

\begin{proof} $(1)$ Let $\alpha$ be a complemented congruence of $A$ such that $\alpha\subseteq Rad(A)$. Assume by absurdum that $\alpha\neq \Delta_A$, hence $\neg\alpha\neq\nabla_A$, so $\neg\alpha\subseteq \phi$, for some maximal congruence $\phi$. On the other hand, we have $\alpha\subseteq Rad(A)\subseteq\phi$, hence we obtain $\nabla_A=\alpha\lor\neg\alpha\subseteq\phi$, contradicting that $\phi\in Max(A)$. Thus it follows that $\alpha=\Delta_A$.

$(2)$ This assertion follows from $(1)$.

\end{proof}

\begin{teorema} Let $\theta$ be a congruence of $A$ such that $\theta\subseteq Rad(A)$. Assume that $\Omega'$ is an orthogonal subset of $B(Con(A/{\theta}))$ which is lifted to a subset $\Omega$ of $B(Con(A))$. Then the set $\Omega$ is unique (w.r.t. the mentioned property) and its elements are orthogonal.

\end{teorema}

\begin{proof} Firstly we shall prove the uniqueness of $\Omega$. Consider that $\Omega'$ is lifted to the subsets $\Omega_1,\Omega_2$ of $B(Con(A))$. Let $\alpha$ be a congruence from $\Omega_1$, so $p_{\theta}^{\bullet}(\alpha)\in \Omega'$. Then there exists a congruence $\beta\in \Omega_2$ such that $p_{\theta}^{\bullet}(\beta)= p_{\theta}^{\bullet}(\alpha)$.

Since $p_{\theta}^{\bullet}|_{B(Con(A))}$ is a Boolean morphism it preserves the difference "$-$", hence $p_{\theta}^{\bullet}(\alpha- \beta)= p_{\theta}^{\bullet}(\alpha)-p_{\theta}^{\bullet}(\beta)=\Delta_{A/{\theta}}$, i.e $((\alpha-\beta)\lor\theta)/{\theta}=(\alpha_1\lor\theta)/{\theta}$. Then $(\alpha-\beta)\lor\theta=\theta$, hence $\alpha-\beta\subseteq\theta\subseteq Rad(A)$. By applying Lemma 7.2(2) we get $\alpha=\beta$, so $\alpha\in \Omega_2$.  We have proven that $\Omega_1\subseteq \Omega_2$. The converse inclusion  $\Omega_2\subseteq \Omega_1$ follows in a similar manner, so $\Omega_1=\Omega_2$.

Consider now two distinct congruences $\alpha_1,\alpha_2\in \Omega$. By using Lemma 7.2(2) it is easy to see that $p_{\theta}^{\bullet}(\alpha_1), p_{\theta}^{\bullet}(\alpha_2)$ are two distinct congruences of $\Omega'$, so they are orthogonal. Thus $([\alpha_1,\alpha_2]\lor\theta)/{\theta}= [(\alpha_1\lor\theta)/{\theta},(\alpha_2\lor\theta)/{\theta}]=(\alpha_1\lor\theta)/{\theta}$, hence $[\alpha_1,\alpha_2]\subseteq\theta\subseteq Rad(A)$. By using Lemma 7.2(1) we obtain $[\alpha_1,\alpha_2]=\Delta_A$, i.e. $\alpha_1,\alpha_2$ are orthogonal congruences.

\end{proof}

\begin{teorema} Let $\theta$ be a congruence of $A$ such that $\theta\subseteq Rad(A)$. If $\theta$ has $CBLP$ then any set of atoms of $B(Con(A/{\theta}))$ can be uniquely lifted to a set of atoms of $B(Con(A))$.
\end{teorema}

\begin{proof} Let $\Omega'$ be a set of atoms of the Boolean algebra $B(Con(A/{\theta}))$. Since $\theta$ has $CBLP$ there exists a set $\Omega\subseteq B(Con(A))$ such that $\Omega'=\{p_{\theta}^{\bullet}(\alpha)|\alpha\in \Omega\}$. The uniqueness of $\Omega$ follows by applying Theorem 7.3.

Now we have to prove that any $\alpha\in \Omega$ is an atom of $B(Con(A))$. Let us consider a congruence $\beta\in B(Con(A))$ such that $\beta\neq \Delta_A$ and $\beta\subseteq\alpha$. If we assume that $p_{\theta}^{\bullet}(\beta)= (\beta\lor\theta)/{\theta}= \Delta_{A/{\theta}}$ then $\beta\subseteq\theta\subseteq Rad(A)$, by using Lemma 7.2(1) we get $\beta=\Delta_A$. This contradiction shows that $p_{\theta}^{\bullet}(\beta)\neq \Delta_{A/{\theta}}$. Since $p_{\theta}^{\bullet}(\beta)\subseteq p_{\theta}^{\bullet}(\alpha)$ and $p_{\theta}^{\bullet}(\alpha)$ is an atom of the Boolean algebra $B(Con(A/{\theta}))$ it follows that $p_{\theta}^{\bullet}(\alpha)=p_{\theta}^{\bullet}(\beta)$.

Let us consider the difference $\alpha-\beta\in B(Con(A))$. Since $p_{\theta}^{\bullet}|_{B(Con(A))}$ is a Boolean morphism it follows that $p_{\theta}^{\bullet}(\alpha-\beta)=p_{\theta}^{\bullet}(\alpha)-p_{\theta}^{\bullet}(\beta)=\Delta_{A/{\theta}}$, hence $\alpha-\beta\subseteq\theta\subseteq Rad(A)$. By applying Lemma 7.2(2) we get $\alpha=\beta$, so $\alpha$ is an atom of $B(Con(A))$.

\end{proof}

\begin{corolar} Let $\theta$ be a congruence of $A$ such that $\theta\subseteq \rho(\Delta_A)$. Then any set of atoms of $B(Con(A/{\theta}))$ can be uniquely lifted to a set of atoms of $B(Con(A))$.
\end{corolar}

\begin{proof} Let $\Omega'$ be a set of atoms of $B(Con(A/{\theta}))$. According to Corollary 5.9, the hypothesis $\theta\subseteq \rho(\Delta_A)$ implies that $\theta$ has $CBLP$. Observing that $\theta\subseteq \rho(\Delta_A)\subseteq Rad(A)$ and applying Theorem 7.4, it follows that $\Omega'$ can be uniquely lifted to a set of atoms of $B(Con(A))$.
\end{proof}


\begin{thebibliography}{200}
\bibitem{Aghajani} M. Aghajani, A. Tarizadeh, Characterization of Gelfand rings, specially clean rings and their dual rings, Results Math.,75:125,2020
\bibitem{Agliano} P. Agliano, Prime spectra in modular varieties, Algebra Universalis, 30, 1993, 581 - 597
\bibitem{Atiyah} M. F. Atiyah, I. G. MacDonald, Introduction to Commutative Algebra, Addison-Wesley Publ. Comp., 1969
\bibitem{Burris} S. Burris, H. P. Sankappanavar, A Course in Universal Algebra, Springer, 1981

\bibitem{Al-Ezeh2} H. Al- Ezeh, Further results on reticulated rings, Acta Math. Hung., 60 (1-2), 1992, 1 - 6
\bibitem{BalbesDwinger} R. Balbes, Ph. Dwinger, Distributive Lattices, Univ. of Missouri Press, 1974
\bibitem{Banaschewski} B. Banaschewski, Gelfand and exchange rings: their spectra in pointfree topology, The Arabian Journal for Science and Engineering, 25, No 2C, 2000, 3 - 22
\bibitem{B1} B. Barania Nia, A. B. Saeid, Study of pseudo $BL$-algebras in view of left Boolean Lifting Property, Applications and Applied Math., 13(1), 2018, 354 - 381
\bibitem{B2} B. Barania Nia, A. B. Saeid, Classes of pseudo $BL$-algebras in view of right Boolean Lifting Property, Trans. of A. Razmadze Math. Institute, 172, 2018, 146 - 163

\bibitem{Birkhoff} G. Birkhoff, Lattice Theory, 3rd ed., AMS Collocquium Publ. Vol. 25, 1967

\bibitem{Burris} S. Burris, H. P. Sankappanavar, A Course in Universal Algebra, Graduate Texts in Mathematics, 78, Springer Verlag, 1881

\bibitem{Cheptea} D. Cheptea, G. Georgescu, C. Mure\c{s}an, Boolean lifting properties for bounded distributive lattices, Scientific Annals of Computer Science, 25, 2015, 29 - 67
\bibitem{DiNola} A. Di Nola, G. Georgescu, L. Leu\c{s}tean, Boolean products of $BL$-algebras, J. Math. Analysis and Applications, 251(1), 2000, 106 - 131
\bibitem{Dickmann} M. Dickmann, N. Schwartz, M. Tressl, Spectral Spaces, Cambridge Univ. Press., 2019

\bibitem{Fresee} R. Fresee, R. McKenzie, Commutator Theory for Congruence Modular Varieties, Cambridge Univ. Press, 1987
\bibitem{Galatos} N. Galatos, P. Jipsen, T. Kowalski, H. Ono, Residuated Lattices: An Algebraic Glimpse at Structural Logics, Studies in Logic and The Foundation of Mathematics, 151, Elsevier, 2007
\bibitem{Filipoiu} A. Filipoiu, G. Georgescu, Compact and Pierce representations of $MV$-algebras, Rev. Roum. Math. Pures Appl.,40(7), 1995, 599 - 618
\bibitem{GG} G. Georgescu, Reticulation functor and the transfer properties, manuscript, 2020
\bibitem{GCM} G. Georgescu, D. Cheptea, C. Mure\c{s}an, Algebraic and topological results on lifting properties in residuated lattices, Fuzzy Sets Systems,  271, 2015, 102-132.

\bibitem{GM2} G. Georgescu, C. Mure\c{s}an, The reticulation of a universal algebra, Scientific Annals of Computer Science, 28, 2018, 67 - 113
\bibitem{GM3} G. Georgescu, C. Mure\c{s}an, Congruence Boolean Lifting Property, J. Multiple Valued Log. Soft Comput., 28(1), 2018, 67 - 113
\bibitem{GM4} G. Georgescu, C. Mure\c{s}an, Congruence extensions in congruence modular varieties, in preparation
\bibitem{GKM} G. Georgescu, L. Kwuida, C. Mure\c{s}an, Functorial properties of the reticulation of a universal algebra, J. Applied Logic, 8(5), 2021, 1123 - 1168

\bibitem{GeorgescuVoiculescu2} G. Georgescu, I. Voiculescu, Some abstract maximal ideal-like spaces, Algebra Universalis, 26, 1989, 90 - 102
\bibitem{f} A. W. Hager, C. M. Kimber, W. Wm. McGovern, Clean unital l-groups, Math. Slovaca, 63, 2013, 979 - 992

\bibitem{Hochster} M. Hochster, Prime ideals structures in commutative rings, Trans.Amer.Math.Soc., 142, 1969, 43 - 60
\bibitem {b} N. A. Immormino, Some Notes on Clean Rings, Bowling Green State University, 2012
\bibitem{Jipsen} P. Jipsen, Generalization of Boolean products for lattice-ordered algebras, Annals Pure Appl. Logic, 161, 2009, 224 - 234
\bibitem{Johnstone} P. T. Johnstone, Stone Spaces, Cambridge Univ. Press, 1982
\bibitem{Jong} A.J. de Jong et al., Stacks Project, see http://stacks.math.columbia.edu
\bibitem{Kermani} N. A. Kermani, E. Eslami, A. B. Saeid, Central lifting property for orthomodular lattices, Math. Slovaca, 70(6), 2020, 1307 - 1316
\bibitem{Kollar} J. Kollar, Congruences and one - element subalgebras, Algebra Universalis, 9, 1979, 266 - 276

\bibitem{Lenzi} G. Lenzi, A. Di Nola, The spectrum problem for abelian $l$ - groups and $MV$ - algebras, Algebra Universalis, 81(3), 2020
\bibitem{g} L. Leu\c{s}tean, Representations of many-valued algebras, Editura Universitara, Bucharest, 2010

\bibitem{c} W. Wm. McGovern, Neat rings, J. Pure Appl. Algebra, 205, 2006, 243 - 265

\bibitem{Muresan} C. Mure\c{s}an, Algebras of many - valued logic. Contributions to the theory of residuated lattices, Ph.D. Thesis, University of Bucharest, 2009
\bibitem{a} W. K. Nicholson, Lifting idempotents and exchange rings, Trans. Amer. Math. Soc., 229, 1977, 268 - 278

\bibitem{Simmons} H. Simmons, Reticulated rings, J. Algebra, 66, 1980, 169 - 192

\bibitem{Tar4} A. Tarizadeh, P.K. Sharma, Characterizing lifting and finiteness of idempotents, ArXiv :1806.03599[math.AC] 26 Jan. 2021
\end{thebibliography}
\end{document}